\documentclass{amsart}
\usepackage{amssymb}
\usepackage{amssymb,amsbsy}
\usepackage{epstopdf}
\usepackage{float}
\usepackage{amsfonts}
\usepackage{amsmath}
\usepackage{amscd}
\usepackage{tikz-cd}
\usepackage{latexsym}
\usepackage{amsthm}
\usepackage{verbatim}
\usepackage{paralist}
\usepackage[utf8]{inputenc}
\usepackage{etoolbox}
\usepackage{enumitem}
\usetikzlibrary{matrix}
\usepackage{hyperref}

\newtheorem{theorem}{Theorem}[section]
\newtheorem{corollary}[theorem]{Corollary}
\newtheorem{lemma}[theorem]{Lemma}
\newtheorem{observation}[theorem]{Observation}

\theoremstyle{definition}
\newtheorem{definition}[theorem]{Definition}

\newtheorem{question}[theorem]{Question}

\theoremstyle{remark}

\DeclareMathOperator{\ran}{range}

\DeclareMathOperator{\dom}{dom}

\newcommand{\st}{\; | \;}
\newcommand{\set}[2]{\left\{#1\st #2 \right\}}
\newcommand{\seq}[2]{\langle #1 \st #2 \rangle}

\DeclareMathOperator{\Add}{\mathcal A\textit{dd}\,}

\newcommand{\forces}{\Vdash}
\newcommand{\rest}{\mathbin{\upharpoonright}}

\newcommand{\SP}[2]{\ensuremath{\mathsf{Split}_{#1}(#2)}}

\newcommand{\Unbounded}{\ensuremath{\mathsf{unbounded}}}
\newcommand{\Stationary}{\ensuremath{\mathsf{stationary}}}
\newcommand{\Nonempty}{\ensuremath{\mathsf{nonempty}}}
\newcommand{\MaxCard}{\ensuremath{[\kappa]^\kappa}}

\newcommand{\Split}[1]{\ensuremath{\SP{#1}{\Unbounded}}}

\newcommand{\calA}{\mathcal{A}}
\newcommand{\calB}{\mathcal{B}}
\newcommand{\calF}{\mathcal{F}}
\newcommand{\calI}{\mathcal{I}}

\newcommand{\calS}{\mathcal{S}}

\newcommand{\sub}{\subseteq}
\newcommand{\To}{\longrightarrow}
\newcommand{\power}[1]{\mathcal{P}(#1)}
\newcommand{\bkappa}{{\bar{\kappa}}}

\newcommand{\kla}[1]{{\langle #1\rangle}}
\DeclareMathOperator{\dunion}{\raisebox{.5mm}{$\bigtriangledown\!$}}
\DeclareMathOperator{\dintersection}{\mathop{\bigtriangleup}}
\newcommand{\card}[1]{{\mathsf{card}(#1)}}
\newcommand{\V}{\mathrm{V}}

\renewcommand{\phi}{\varphi}
\newcommand{\cf}{\mathrm{cf}}

\newcommand{\balpha}{{\bar{\alpha}}}
\newcommand{\bbeta}{{\bar{\beta}}}

\newcommand{\bdelta}{{\bar{\delta}}}

\newcommand{\bA}{{\bar{A}}}
\newcommand{\bard}{{\bar{d}}}

\newcommand{\leer}{\emptyset}
\newcommand{\ohne}{\setminus}

\newcommand{\sn}{\mathfrak{s}} 

\newcommand{\tA}{{\tilde{A}}}

\newcommand{\bbM}{{\mathbb{M}}}

\newcommand{\Q}{\mathbb{Q}}

\renewcommand{\P}{\mathbb{P}}

\DeclareMathOperator{\spread}{spread}

\title{Split Principles}
\author{Gunter Fuchs}
\address{The College of Staten Island (CUNY)\\2800 Victory Blvd.~\\Staten Island, NY 10314}
\address{The Graduate Center (CUNY)\\365 5th Avenue, New York, NY 10016}
\email{gunter.fuchs@csi.cuny.edu}
\urladdr{www.math.csi.cuny.edu/~fuchs}

\author{Kaethe Minden}
\address{Bard College at Simon's Rock\\84 Alford Rd\\Great Barrington, MA 01230}
\email{kminden@simons-rock.edu}
\urladdr{https://kaetheminden.wordpress.com/}

\date{\today}



\begin{document}

\begin{abstract}
We introduce the split principles and show that they bear tight connections to large cardinal properties such as inaccessibility, weak compactness, subtlety, almost ineffability and ineffability, as well as classical combinatorial objects such as Aronszajn trees, Souslin trees or square principles. We exhibit correspondences between certain split principles and splitting numbers at uncountable cardinals.
\end{abstract}  	

\maketitle

\section{Introduction}

A version of the split principle was first considered by the first author, Gitman, and Hamkins in the course of their work on \cite{FGH:IncomparableModels}, the intended use being the construction of ultrafilters with certain desirable properties that the authors ultimately were able to do without. But later, the first author observed that the split principle is an ``anti-large-cardinal axiom'' which characterizes the failure of a regular cardinal to be weakly compact. In the present paper, we consider several versions of the principle that provide uniform combinatorial characterizations of the failure of various large cardinal properties, and that have tight connections to other set-theoretic concepts like splitting families and trees.

In its original form, the split principle for a cardinal $\kappa$ postulates the existence of a $\kappa$-list $\vec d = \seq{d_\alpha}{\alpha<\kappa}$ that ``splits" every unbounded subset $A$ of $\kappa$ into two unbounded subsets of $A$, meaning that there is one ordinal $\beta$ such that for unboundedly many $\alpha\in A$, we have that $\beta$ belongs to $d_\alpha$, but also, for unboundedly many $\alpha\in A$, $\beta$ belongs to the relative complement $\alpha\ohne d_\alpha$. By varying the form of the list allowed (it may be a thin list, or just a sequence of subsets of some cardinal $\tau$), the kinds of sets that are split (unbounded subsets of $\kappa$, stationary subsets of $\kappa$; large sets in some sense), and the kinds of sets split into (thus, for example one could consider splitting stationary sets into unbounded sets), we obtain a host of natural split principles.

Given a $\kappa$-sequence $\vec{d}$ \emph{of subsets of $\tau$,} we may imagine this sequence as being given by the sequence of vertical sections of a subset of $\kappa\times\tau$. We can also form the set $\calF$ of its horizontal sections, a family of subsets of $\kappa$. The most immediate connection to another area of set-theoretic research is revealed by the  realization that $\vec{d}$ is a split sequence iff the family $\calF$ is a \emph{splitting family}: for every unbounded set $A\sub\kappa$ there is an $X\in\calF$ such that $A\cap X$ and $A\ohne X$ are both unbounded, and this equivalence holds in more generality for any families of large sets instead of the unbounded sets.

Thus, many, but not all of our results can be viewed as talking about splitting families and splitting numbers (splitting sequences of subsets of a fixed $\tau$ can always be translated in this way, but $\kappa$-lists cannot.) However, considering the vertical sections rather than the horizontal ones is a crucial change in perspective. Most notably, it makes a difference if the sequence $\vec{d}$ is thin, meaning that for every $\beta<\tau$, the set $\{d_\alpha\cap\beta\st\alpha<\kappa\}$ has cardinality less than $\kappa$. The corresponding property of a splitting family has not been considered before, to our knowledge. It was the ability to go back and forth between splitting families and split sequences that made this a natural property, and this opens up many other connections as well. Thus, we reveal close connections between split principles and certain large cardinal properties, combinatorial principles such as $\square$, the existence of certain trees, etc. Here is a summary of the contents of this work by section.

In Section \ref{sec:OGSplit}, we introduce the original split principle in detail and focus on its connection to trees. We show its equivalence at a regular cardinal $\kappa$ to the existence of a sequential tree $T\sub{}^{{<}\kappa}2$ of height $\kappa$ with no cofinal branch, and we show that the existence of a $\kappa$-Aronszajn tree is equivalent to the \emph{thin} split principle at $\kappa$.

In Section \ref{sec:SplittingSubsetsOfKappaAndSplittingNo} we show that if the notions of largeness used are reasonable (namely, the largeness of the sets split into is determined by a tail of the sets), then the failure of the split principle at $\kappa$ says that the corresponding splitting number is larger than $\kappa$. We also introduce the split principles concerning $\kappa$-sequences of subsets of some cardinal $\tau$ rather than $\kappa$-lists and show that a principle of this type holds iff there is a splitting family of size $\tau$.

Section \ref{sec:LargeCardinals} contains our characterizations of large cardinal properties. First, in \ref{subsec:inacc} we characterize inaccessibility, and observe that regularity can also be expressed using split principles. In \ref{subsec:wc} we characterize weak compactness by the failure of certain split principles of unbounded sets. In \ref{subsec:Ineffables} we characterize ineffability and almost ineffability by the failure of split principles centered around stationary sets. Interestingly, the split principle characterizing almost ineffabilty talks about splitting unbounded sets into nonempty sets by $\kappa$-lists and does not translate to a statement about splitting numbers. We also show that the ineffable tree property of \cite{Weiss:Diss} is characterized by the failure of a thin split principle. Since splitting unbounded sets and splitting stationary sets were crucial in the characterizations so far, we include some results on the relationships between the corresponding splitting numbers in this section, and we end by showing how to increase the most relevant ones by forcing over a model with a supercompact cardinal.
Finally, in \ref{subsec:Subtles}, we present a characterization of subtlety, as well as a characterization of the existence of a $\kappa$-Souslin tree.

Next, in Section \ref{sec:singular} we consider the split principles, splitting families and splitting numbers at a singular cardinal $\kappa$. Roughly speaking, we develop methods to transfer split sequences between $\kappa$ and $\cf(\kappa)$ that work in the case of stationary sets and unbounded sets, but work only in one direction in the case of the subsets of maximal cardinality, unless the sequence in question is thin. This is why we are able to resolve questions about the splitting numbers of the unbounded and the stationary subsets of a singular $\kappa$ that are open about the sets of maximal cardinality, unless one restricts attention to thinness. A main result is that is consistent that the splitting number of the stationary subsets of a singular cardinal $\kappa$ of uncountable cofinality is greater than $\kappa$, and the same is true of the splitting number of the unbounded subsets of a singular $\kappa$, and of the ``thin'' splitting number of $[\kappa]^\kappa$. That is, the corresponding split principles fail at $\kappa$, and the thin split principle for $[\kappa]^\kappa$ also fails at $\kappa$. We also have many results about comparisons between various splitting numbers at a singular cardinal.


\section{The original split principle}
\label{sec:OGSplit}
Before introducing split principles in complete generality, let us first introduce the original split principle, which characterizes the failure of having a branch through any tree in a kind of ``positive" way. Indeed, split principles on regular cardinals can be thought of loosely as anti-branch properties.

Let $\kappa$ be a cardinal.
We shall use the terminology of \cite{Weiss:Diss}, and refer to a sequence of the form $\seq{ d_\alpha }{ \alpha < \kappa }$ as a \emph{$\kappa$-list} if for all $\alpha<\kappa$, $d_\alpha \subseteq \alpha$.

\begin{definition}
\label{definition:KappaSplitPrinciple}
Let $\kappa$ be a cardinal. The principle $\Split{\kappa}$ says that there is a $\kappa$-list $\vec{d}$ that \emph{splits unbounded sets,} meaning that for every $A \subseteq \kappa$ unbounded, there is a $\beta<\kappa$ such that both $$A^+_{\beta,\vec{d}}=A^+_\beta=\{\alpha\in A\st\beta\in d_\alpha\} \text{ and } A^-_{\beta,\vec{d}}=A^-_\beta=\{\alpha\in A\st\beta\in\alpha\ohne d_\alpha\}$$ are unbounded. In this case, we also say that $\beta$ splits $A$ into unbounded sets with respect to $\vec{d}$.
\end{definition}

\begin{figure}[h]
\begin{tikzpicture}
 \draw[->, densely dotted] (-3,-1.3)--(3,-1.3);
 \node at (-3,-1.3) {$-$}; \node [left] at (-3,-1.3) {$\beta$};

\draw[->] (-3,-3)--(3,-3);
\draw[->] (-3,-3)--(-3,3);
\draw[->, dotted] (-3,-3)--(3,3);

\node [left] at (-3,3) {$\kappa$};
\node [below] at (3,-3) {$\kappa$};

\node at (2,-3) {$|$}; \node [below] at (2,-3.1) {$\gamma \in A_\beta^-$};
\draw[fill, color=gray] (-2,-2.5) circle [x radius = 0.1cm, y radius=0.2cm];

\node at (-3,2) {$-$}; \node [left] at (-3,2) {$\gamma$};
\node [above]  at (2,2) {$d_\gamma$};

\draw[fill, color=gray] (-1,-2.2) circle [x radius = 0.15cm, y radius=0.7cm];
\node at (-1,-3) {$\circ$};

\node at (0,-3) {$\bullet$}; \node at (0,-3) {$|$}; \node [below] at (0,-3.1) {$\alpha \in A_\beta^+$};
\draw [fill, color=gray] (0,-1) circle [x radius=0.15cm, y radius=0.7cm];
\draw[fill, color=gray] (0,-2.5) circle [x radius = 0.1cm, y radius=0.3cm];
\node at (-3,0) {$-$}; \node [left] at (-3,0) {$\alpha$};
\node [above] at (0,0) {$d_\alpha$};

\node at (-0.5,-3) {$\circ$};

\draw [fill, color=gray] (1,-1) circle [x radius=0.2cm, y radius=1.8cm];
\node  at (1,-3) {$\bullet$};

\draw[fill, color=gray] (2,1.5) circle [x radius = 0.1cm, y radius=0.3cm];
\node[color=gray] at (2, 1) {$\bullet$};
\node at (2,-3) {$\circ$};

\draw [fill, color=gray] (2.5,-1.5) circle [x radius=0.13cm, y radius=0.5cm];
\node at (2.5,-3) {$\bullet$};

\node at (2.7,-3) {$\circ$};
\end{tikzpicture}

\caption{Here $\beta$ splits $A$ into unbounded sets with respect to a $\kappa$-list $\vec{d}$.}
\label{figure:splitkappa}
\end{figure}

In Figure \ref{figure:splitkappa}, an attempt to picture a split sequence visually is given. Some of the elements of a $\kappa$-list $\vec d$ are represented, by blobs. The focus is on indices of $\vec d$ that lie in $A$, an unbounded subset of $\kappa$. For any $\beta<\kappa$, each element $\alpha$ of $A\setminus(\beta+1)$ falls into either $A^+_\beta$ or $A^-_\beta$, depending on whether $\beta$ is an element of $d_\alpha$ or not. So $A$ is the disjoint union $(A\cap(\beta+1))\cup A^+_\beta\cup A^-_\beta$. We say that $\beta$ splits $A$ into unbounded sets with respect to $\vec d$ if both $A^+_\beta$ and $A^-_\beta$ are unbounded subsets of $\kappa$.




Recall that a regular cardinal $\kappa$ is \emph{weakly compact} if $\kappa$ is inaccessible and the tree property holds at $\kappa$, that is, every $\kappa$-tree has a cofinal branch, where a $\kappa$-tree is a tree of height $\kappa$ all of whose levels have size less than $\kappa$.
We will show in section \ref{subsec:wc} that a regular cardinal $\kappa$ is weakly compact if and only if $\SP{\kappa}{\Unbounded}$ fails. To this end, we will first define the analogue of the tree property for $\kappa$-lists.

\begin{definition}
\label{def:CofinalBranchOfAKappaList}
A $\kappa$-list $\vec d = \seq{ d_\alpha }{ \alpha < \kappa }$ has a \emph{cofinal branch}, or has a $\kappa$-branch, so long as there is
a $b \subseteq \kappa$ such that for all $\gamma < \kappa$ there is an $\alpha \geq \gamma$ such that
$d_\alpha \cap \gamma = b \cap \gamma.$


Given a $\kappa$-list $\vec d = \seq{ d_\alpha }{ \alpha < \kappa }$, for each $\alpha < \kappa$ let $d^c_\alpha: \alpha \to 2$ denote the characteristic function of $d_\alpha$. The sequential tree corresponding to the $\kappa$-list is given by
$T_{\vec d} = \set{ d^c_\alpha \rest \beta }{\beta\le\alpha<\kappa}$,
and the tree ordering is set inclusion. As is customary, we refer to a function
$b: \kappa \to 2$ as a (cofinal) branch through $T_{\vec d}$ if for all $\gamma < \kappa$, $b\rest\gamma\in T_{\vec{d}}$, which means that for every $\gamma<\kappa$, there is an $\alpha\ge\gamma$ such that
$b \rest \gamma = d^c_\alpha \rest \gamma.$
 \end{definition}

\begin{observation}
\label{observation:ListBranchCorrespondsToTreeBranch}
A $\kappa$-list $\vec d$ has a cofinal branch if and only if the corresponding tree $T_{\vec{d}}$ has a cofinal branch.
\end{observation}

It turns out that for regular $\kappa$, the properties of a $\kappa$-list of splitting unbounded sets and having a cofinal branch are complementary.

\begin{theorem}
\label{theorem:ListSplitsIffItHasNoBranch}
Let $\kappa$ be regular, and let $\vec{d}$ be a $\kappa$-list. The following are equivalent:
\begin{enumerate}[label=\textnormal{(\arabic*)}]
  \item $\vec{d}$ witnesses $\SP{\kappa}{\Unbounded}$.
  \item $\vec{d}$ does \emph{not} have a cofinal branch.
\end{enumerate}
\end{theorem}

\begin{proof}
(1)$\implies$(2): Assume towards a contradiction that $\vec{d}$ has a cofinal branch $b \subseteq \kappa$, and that $\vec d$ splits unbounded sets. We will first define a function $f:\kappa \To \kappa$ as follows:
\[f(\gamma) =\text{the least $\alpha \geq \gamma$ such that $b\cap\gamma = d_\alpha \cap \gamma.$}\]
Note that $f$ is weakly increasing, thus letting $A = f``\kappa$, $A$ is unbounded in $\kappa$. So there is $\beta < \kappa$ which splits $A$ (with respect to $\vec{d}$), i.e., both of the sets
$A_\beta^+ = \set{ f(\gamma) \in A }{ \beta \in d_{f(\gamma)} }$ and $A_\beta^- = \set{ f(\gamma) \in A }{ \beta \in f(\gamma) \setminus d_{f(\gamma)} }$
are unbounded in $\kappa$. There are two cases.
		 	
\emph{Case 1: $\beta \notin b$.} Since $A_\beta^+$ is unbounded, we may choose  $f(\gamma) \in A_\beta^+$ satisfying $f(\gamma)>f(\beta)$. By the weak monotonicity of $f$, it follows that $\gamma>\beta$. Then $\beta \in d_{f(\gamma)}$ by the definition of $A_\beta^+$ and it follows that
			$\beta \in d_{f(\gamma)} \cap \gamma= b \cap \gamma,$
contradicting the assumption that $\beta \notin b$.
		
\emph{Case 2: $\beta \in b$.} We may run a similar argument to the previous case in order to obtain a contradiction. In this case, we use that $A_\beta^-$ is unbounded to choose $f(\gamma) \in A_\beta^-$ satisfying $f(\gamma) > f(\beta)$, so that $\gamma>\beta$, and get the contradiction that $\beta\notin d_{f(\gamma)}$ while $\beta \in b\cap\gamma=d_{f(\gamma)} \cap \gamma$.
	
$\neg(1)$$\implies$$\neg(2)$: Assume that $\vec{d}$ does not split unbounded sets. We will show that then $\vec{d}$ has a cofinal branch. Let $A \subseteq \kappa$ be unbounded such that no $\beta < \kappa$ splits $A$ (with respect to $\vec{d}$). Thus, for each $\beta < \kappa$, exactly one of $A^+_\beta$ or $A^-_\beta$ is bounded in $\kappa$, since $A\ohne(\beta+1)= A_\beta^+ \cup A_\beta^-$ (so since $A$ is unbounded, it can't be that both $A^+_\beta$ and $A^-_\beta$ are bounded). Now we may define our branch $b\subseteq\kappa$ as follows:
\[\beta \in b \iff A_\beta^- \text{ is bounded in $\kappa$. }\]
To see that this works, note that for each $\beta < \kappa$ there is a least $\xi_\beta < \kappa$ such that either:
		$\beta \in d_\delta$ for all $\delta \in A \setminus \xi_\beta$  or $\beta \notin d_\delta$ for all $\delta \in A \setminus \xi_\beta$.

Letting $\gamma < \kappa$ be arbitrary, using the fact that $\kappa$ is regular, there is an $\alpha \in A$ such that
$\alpha > \sup_{\beta < \gamma} \xi_\beta.$
It follows that $b \cap \gamma = d_\alpha \cap \gamma$.
To see this, let $\beta < \gamma$. We have two cases.

\emph{Case 1: $\beta \in b$.} Then $A^-_\beta$ is bounded, so $\beta\in d_\alpha$, since $\alpha\in A\setminus\xi_\beta$.

\emph{Case 2: $\beta \notin b$.} Then $A^+_\beta$ is bounded, so $\beta \notin d_\alpha$, since $\alpha \in A \setminus \xi_\beta$.

So $b$ is a cofinal branch.	
\end{proof}

\begin{corollary}
\label{corollary:TreeCharacterizationIfKappaIsRegular}
Let $\kappa$ be regular. Then $\SP{\kappa}{\Unbounded}$ holds iff there is a sequential tree $T\sub{}^{{<}\kappa}2$ of height $\kappa$ that has no cofinal branch.
\end{corollary}

\begin{proof}
For the direction from right to left, let $T$ be a sequential tree $T\sub{}^{{<}\kappa}2$ of height $\kappa$ that has no cofinal branch. For each $\alpha<\kappa$, let $s_\alpha$ be a node at the $\alpha$-th level of $T$, i.e., a sequence $s_\alpha:\alpha\To 2$ with $s_\alpha\in T$. Let $d_\alpha$ be the sequence $s_\alpha$, viewed as a subset of $\alpha$, i.e., $d_\alpha=\{\gamma<\alpha\st s_\alpha(\gamma)=1\}$. In other words, $s_\alpha=d^c_\alpha$. Then $T_{\vec{d}}\sub T$, and so, since $T$ does not have a cofinal branch, $T_{\vec{d}}$ has no cofinal branch, which means, by Observation \ref{observation:ListBranchCorrespondsToTreeBranch}, that $\vec{d}$ has no cofinal branch, and this is equivalent to saying that $\vec{d}$ splits unbounded sets, by Theorem \ref{theorem:ListSplitsIffItHasNoBranch}. So $\SP{\kappa}{\Unbounded}$ holds.

For the converse, let $\vec{d}$ be a $\SP{\kappa}{\Unbounded}$-sequence. By Theorem \ref{theorem:ListSplitsIffItHasNoBranch}, $\vec{d}$ has no cofinal branch. So by Observation \ref{observation:ListBranchCorrespondsToTreeBranch}, $T_{\vec{d}}$ does not have a cofinal branch, so $T_{\vec{d}}$ is as wished.
\end{proof}


Note that a sequential tree $T\sub{}^{{<}\kappa}2$ of height $\kappa$ without a cofinal branch is not necessarily a $\kappa$-Aronszajn tree, as it may have levels of size $\kappa$. Of course, if $\kappa$ is inaccessible, then such a $T$ is automatically Aronszajn. However, there is a natural class of $\kappa$-lists which give rise to $\kappa$-trees. These are called \emph{thin}, and they allow us to formulate a more stringent split principle that corresponds directly to the existence of $\kappa$-Aronszajn trees.

\begin{definition}[{\cite[Def.~1.2.4]{Weiss:Diss}}]
\label{definition:ThinSplit}
For a cardinal $\kappa$, a $\kappa$-list $\vec{d}=\seq{d_\alpha}{\alpha<\kappa}$ is \emph{thin} if for every $\beta<\kappa$, the set $\{d_\alpha\cap\beta\st\alpha<\kappa\}$ has cardinality less than $\kappa$.

The principle \emph{thin $\Split{\kappa}$} says that $\Split{\kappa}$ holds, as witnessed by a \emph{thin} $\kappa$-list.
\end{definition}

The motivation for the concept of a thin list is the following easy observation.

\begin{observation}
\label{obs:ThinListsCorrespondToKappaTrees}
Let $\kappa$ be an infinite cardinal, and let $\vec{d}$ be a $\kappa$-list. Then $\vec{d}$ is thin iff $T_{\vec{d}}$ is a $\kappa$-tree.
\end{observation}

We obtain the following version of Theorem \ref{theorem:ListSplitsIffItHasNoBranch}.

\begin{theorem}
\label{theorem:ThinListSplitsIffItsTreeIsAronszajn}
Let $\kappa$ be regular, and let $\vec{d}$ be a $\kappa$-list. The following are equivalent:
\begin{enumerate}[label=\textnormal{(\arabic*)}]
  \item $\vec{d}$ is a thin $\SP{\kappa}{\Unbounded}$ sequence.
  \item $T_{\vec{d}}$ is a $\kappa$-Aronszajn tree.
\end{enumerate}
\end{theorem}

\begin{proof}
Immediate, given Theorem \ref{theorem:ListSplitsIffItHasNoBranch}.
\end{proof}

Hence, we arrive at:

\begin{theorem}
\label{thm:ThinSplitIffAronszajnTreeExistsAtRegularKappa}
Let $\kappa$ be a regular cardinal. Then the following are equivalent:
\begin{enumerate}[label=\textnormal{(\arabic*)}]
\item
\label{item:ThinSplitHolds}
Thin $\Split{\kappa}$ holds.
\item
\label{item:AT}
There is a $\kappa$-Aronszajn tree.
\end{enumerate}
\end{theorem}

\begin{proof}
\ref{item:ThinSplitHolds}$\implies$\ref{item:AT}: Given a thin $\Split{\kappa}$-sequence $\vec{d}$, $T_{\vec{d}}$ is a $\kappa$-Aronszajn tree, by Theorem \ref{theorem:ThinListSplitsIffItsTreeIsAronszajn}.

\ref{item:AT}$\implies$\ref{item:ThinSplitHolds}: Suppose $S^0$ is a $\kappa$-Aronszajn tree. By \cite[Lemma III.5.27 \& Ex.~III.5.28]{Kunen:SetTheoryNew}, there is then a well-pruned, Hausdorff $\kappa$-Aronszajn tree, call it $S^1$.
By \cite[Ex.~III.5.42]{Kunen:SetTheoryNew}, $S^1$ is isomorphic to a subtree $S^2$ of some sequential tree $B^{{<}\kappa}$. Since the cardinality of $S^1$ is $\kappa$, we may assume that $B=\kappa$. By adding intermediate nodes to $S^2$, we may construct from $S^2$ a $\kappa$-Aronszajn tree $T$ which is a subtree of $2^{{<}\kappa}$.
Now the construction of the proof of the backwards direction of Corollary \ref{corollary:TreeCharacterizationIfKappaIsRegular} yields a $\Split{\kappa}$-sequence $\vec{d}$ such that $T_{\vec{d}}\sub T$. This implies that $T_{\vec{d}}$ is a $\kappa$-tree, and hence that $\vec{d}$ is thin, by Observation \ref{obs:ThinListsCorrespondToKappaTrees}.
\end{proof} 
\section{Splitting subsets of $\kappa$ and the Splitting Number}
\label{sec:SplittingSubsetsOfKappaAndSplittingNo}

Let us now generalize the original split principle by introducing some parameters that can be changed. First, we add some freedom by allowing different notions of largeness; for subsets of $\kappa$ to split and split into.

\begin{definition}
\label{definition:KappaSplitPrincipleWithParameters}
Let $\calA$ and $\calB$ be families of subsets of the cardinal $\kappa$. The principle $\SP{\kappa}{\calA,\calB}$ says that there is a $\kappa$-list $\vec{d}$ that \emph{splits $\calA$ into $\calB$,} meaning that for every $A\in\calA$, there is a $\beta<\kappa$ such that both $A^+_{\beta,\vec{d}}=A^+_\beta=\{\alpha\in A\st\beta\in d_\alpha\}$ and $A^-_{\beta,\vec{d}}=A^-_\beta=\{\alpha\in A\st\beta\in\alpha\ohne d_\alpha\}$ are in $\calB$. In this case, we also say that $\beta$ splits $\calA$ into $\calB$ with respect to $\vec{d}$.

We abbreviate $\SP{\kappa}{\calA,\calA}$ by $\SP{\kappa}{\calA}$ .

The collections of all unbounded, all stationary, all $\kappa$-sized and all nonempty subsets of $\kappa$ are denoted by $\Unbounded_\kappa$, $\Stationary_\kappa$, $\MaxCard$ and $\Nonempty_\kappa$, respectively. When it is clear from context, we will drop the subscript $\kappa$.

We say that $\calA$ is \emph{closed under supersets} if whenever $A\in\calA$ and $A\sub B\sub\kappa$, then $B\in\calA$.
\end{definition}

The idea is that $\calA$ and $\calB$ are collections of ``large'' sets (in particular, these families of sets will usually be closed under supersets), and $\SP{\kappa}{\calA,\calB}$ says that there is one $\kappa$-list that can split any set that's large in the sense of $\calA$ into two disjoint sets that are large in the sense of $\calB$, in the uniform way described in the definition. In most cases, the families we consider will be related to a natural ideal on $\kappa$. In this context, when $\mathcal{I}\sub\power{\kappa}$ is an ideal on $\kappa$, we say that $\mathcal{I}$ is proper if $\kappa\notin\mathcal{I}$, it is $\tau$-complete, for some cardinal $\tau$, if whenever $\theta<\tau$ and $\seq{N_\xi}{\xi<\theta}$ is a sequence of members of $\mathcal{I}$, then $\bigcup_{\xi<\theta}N_\xi\in\mathcal{I}$. We denote by $\mathcal{I}^+$ the collection of $\mathcal{I}$-positive sets, that is, $\power{\kappa}\ohne\mathcal{I}$. The collections we are interested in will often be of this form. (All those mentioned in Definition \ref{definition:KappaSplitPrincipleWithParameters} are.)

\begin{observation}
\label{obs:Basics}
Let $\kappa$ be a cardinal, let $\calA,\calA',\calB,\calB'\sub\power{\kappa}$, and suppose that $\SP{\kappa}{\calA,\calB}$ holds.
\begin{enumerate}[label=\textnormal{(\arabic*)}]
  \item
  \label{item:TrivialTuning}
  If $\calA'\sub\calA$, $\calB\sub\calB'$, then $\SP{\kappa}{\calA',\calB'}$ holds.
  \item
  \label{item:AcontainedinB}
  If $\calB$ is closed under supersets, then $\calA\sub\calB$.
\end{enumerate}
\end{observation}

\begin{proof}
Part \ref{item:TrivialTuning} is obvious. For \ref{item:AcontainedinB}, fix a $\kappa$-list $\vec{d}$ witnessing that $\SP{\kappa}{\calA,\calB}$ holds. Let $A\in\calA$, and pick $\beta<\kappa$ such that $A^+_{\beta,\vec{d}}, A^-_{\beta,\vec{d}}\in\calB$. Since $A^+_{\beta,\vec{d}}\sub A$, it follows that $A\in\calB$.
\end{proof}

Hence, focusing on families that are closed under supersets, and fixing such an $\calA\sub\power{\kappa}$, the \emph{strongest possible} split principle at $\kappa$ with $\calA$ as the first parameter is the ``symmetric principle'' $\SP{\kappa}{\calA,\calA}$, which we write as $\SP{\kappa}{\calA}$. At the other end of the spectrum, the \emph{weakest possible} principle would be $\SP{\kappa}{\calA,\power{\kappa}}$, but it is not of interest, because it holds trivially. The weakest principle of interest results from removing one set from $\power{\kappa}$, namely $\leer$: $\SP{\kappa}{\calA,\Nonempty}$. Thus principles of the form $\SP{\kappa}{\calA}$ and $\SP{\kappa}{\calA,\Nonempty}$ figure prominently in this work.

We have already seen that another parameter that can be adjusted concerns the kinds of ``lists'' that are used to split, when we considered the \emph{thin} $\Split{\kappa}$ principle. In the following, we allow sequences of sets of ordinals which aren't $\kappa$-lists. It will turn out that the resulting principles bear a very close connection to the concept of splitting families.

\begin{definition}
\label{definition:GeneralNonsenseOnKappaSplitting}
Let $\kappa$ and $\tau$ be cardinals, and let $\calA$ and $\calB$ be families of subsets of $\kappa$. The principle $\SP{\kappa,\tau}{\calA,\calB}$ says: there is a $\kappa$-sequence $\vec{d}=\seq{d_\alpha}{\alpha<\kappa}$ of subsets of $\tau$ such that for every $A\in\calA$, there is an ordinal $\beta<\tau$ such that the sets $\tA^+_\beta=\tA^+_{\beta,\vec{d}}=\{\alpha\in A\st\beta\in d_\alpha\}$ and $\tA^-_\beta=\tA^-_{\beta,\vec{d}}=\{\alpha\in A\st\beta\notin d_\alpha\}$ belong to $\calB$. Such a sequence is called a \emph{$\SP{\kappa,\tau}{\calA,\calB}$ sequence.} 

If $\calA=\calB$, we show only one argument in the notation. Thus, $\SP{\kappa,\tau}{\calA,\calA}$ is $\SP{\kappa,\tau}{\calA}$.


In this context, a family $\calF\sub\power{\kappa}$ is an \emph{$(\calA,\calB)$-splitting family for $\kappa$} if for every $A\in\calA$, there is an $S\in\calF$ such that both $A\cap S$ and $A\ohne S$ belong to $\calB$. The $(\calA,\calB)$-\emph{splitting number}, denoted $\sn_{\calA,\calB}(\kappa)$, is the least size of an $(\calA,\calB)$-splitting family. We write $\sn_{\calA}(\kappa)$ for $\sn_{\calA,\calA}(\kappa)$. $\sn(\kappa)$ stands for $\sn_{\MaxCard}(\kappa)$.
\end{definition}

We refer to the split principles of the form $\SP{\kappa,\tau}{\calA,\calB}$ as \emph{two cardinal split principles}. The obvious version of Observation \ref{obs:Basics} applies to the principles of the form $\SP{\kappa,\tau}{\calA,\calB}$ as well. In addition, increasing $\tau$ weakens the principle. Note that if $\calB$ is closed under supersets and $\calF$ is an $(\calA,\calB)$-splitting family, then so is $\calF\cap\calB$, because given $A\in\calA$, let $S\in\calF$ be such that $A\cap S\in\calB$. Then $A\cap S\sub S$, so $S\in\calB$, i.e., $S\in\calF\cap\calB$. Thus, if $\calB$ is closed under supersets, which is always true in the cases we consider, then we can always assume that an $(\calA,\calB)$-splitting family is a subfamily of $\calB$.
This is in line with the usual definition of the well-known cardinal characteristic $\sn=\sn(\omega)$, the least cardinal $\theta$ such that there is a family $\calF$ of \emph{infinite subsets} of $\omega$ that splits infinite sets into infinite sets.

We will now explore the relationships to the previously introduced split principles. First, we will observe that in many cases, the two cardinal split principle at $\kappa,\kappa$ is equivalent to the one cardinal split principle at $\kappa$.

\begin{observation} \label{observation:IndependentOfInitialSegmentsGivesSplitEquivalentToGeneralSplit}
Let $\kappa$ be a cardinal, and let $\calA$ and $\calB$ be collections of subsets of $\kappa$, where $\calB$ is a \emph{tail set,} that is, for all $B\sub\kappa$ and all $\beta<\kappa$, $B\in\calB$ iff $B\ohne\beta\in\calB$. Then $\SP{\kappa,\kappa}{\calA,\calB}$ is equivalent to
$\SP{\kappa}{\calA,\calB}$.
\end{observation}

\begin{proof} If $\vec{e}$ is a $\SP{\kappa}{\calA,\calB}$-sequence, then it is also a $\SP{\kappa,\kappa}{\calA,\calB}$-sequence. This is because for $A\sub\kappa$ and $\beta<\kappa$, $A^+_{\beta,\vec{e}}=\tilde{A}^+_{\beta,\vec{e}}$, and $A^-_{\beta,\vec{e}}=\{\alpha\in A\st\beta\in\alpha\ohne e_\alpha\}=\{\alpha\in A\st\beta\notin e_\alpha\}\cap\{\alpha\in A\st \beta\in\alpha\}=\tilde{A}^-_{\beta,\vec{e}}\ohne(\beta+1)$. Thus, given $A\in\calA$, we can find a $\beta<\kappa$ such that both $\tilde{A}^+_{\beta,\vec{e}}=A^+_{\beta,\vec{e}}$ and $A^-_{\beta,\vec{e}}$ belong to $\calB$. But then $\tilde{A}^-_{\beta,\vec{e}}$ is also in $\calB$, because $\calB$ is a tail set.

For the other direction, let $\vec{d}$ be a $\SP{\kappa,\kappa}{\calA,\calB}$ sequence. Define the $\kappa$-list $\vec{e}$ by $e_\alpha=d_\alpha\cap\alpha$. It follows that $\vec{e}$ is a $\SP{\kappa}{\calA,\calB}$ sequence, because if $A\in\calA$ and $\beta<\kappa$ is such that $\tA^+_{\beta,\vec{d}}$ and $\tA^-_{\beta,\vec{d}}$ are in $\calB$ by $\SP{\kappa,\kappa}{\calA,\calB}$, then $A^+_{\beta,\vec{e}}=\tA^+_{\beta,\vec{d}}\ohne(\beta+1)$, and similarly, $A^-_{\beta,\vec{e}}=\tA^-_{\beta,\vec{d}}\ohne(\beta+1)$. It follows from our assumption on $\calB$ that $A^+_{\beta,\vec{e}}$ and $A^-_{\beta,\vec{e}}$ are in $\calB$.
\end{proof} 

Note that $\Unbounded$, $\Stationary$ and $\MaxCard$ all are tail sets, but $\Nonempty$ is not. Usually, the assumptions of the following observation will be satisfied, so that sequences witnessing the two cardinal split principles can be chosen to be $\kappa$-lists.

\begin{observation}
\label{obs:ReductionToKappaLists}
Let $\kappa$ be a cardinal, $\tau<\kappa$, and let $\calA$ and $\calB$ be collections of subsets of $\kappa$, where $\calA$ is a tail set and $\calB$ is closed under supersets (i.e., if $B\in\calB$ and $B\sub B'\sub\kappa$, then $B'\in\calB$). Then $\SP{\kappa,\tau}{\calA,\calB}$ is equivalent to the principle stating that there is a $\SP{\kappa,\tau}{\calA,\calB}$-sequence which is a $\kappa$-list.
\end{observation}

\begin{proof}
For the substantial direction, suppose that $\vec{d}$ is a $\SP{\kappa,\tau}{\calA,\calB}$-sequence. Let $e_\alpha=d_\alpha\cap\alpha$, for $\alpha<\kappa$. Note that since $d_\alpha\sub\tau$, we have that $e_\alpha=d_\alpha$ for all $\alpha\in[\tau,\kappa)$. Since $\vec{e}$ is a $\kappa$-list, it remains to show that it witnesses $\SP{\kappa,\tau}{\calA,\calB}$. To this end, let $A\in\calA$ be given. Since $\calA$ is a tail set, then $A'=A\ohne\tau\in\calA$. Since $\vec{d}$ witnesses $\SP{\kappa,\tau}{\calA,\calB}$, we can choose $\beta<\tau$ so that $\tilde{A'}^+_{\beta,\vec{d}}$ and $\tilde{A'}^-_{\beta,\vec{d}}$ are in $\calB$.

But observe that $\tilde{A'}^+_{\beta,\vec{d}}=A'^+_{\beta,\vec{e}}$: $\alpha\in\tilde{A'}^+_{\beta,\vec{d}}$ iff $\alpha\in A'$ and $\beta\in d_\alpha$. Since $A'\cap\tau=\leer$, $\alpha\ge\tau$, so $d_\alpha=e_\alpha$. Thus, this is equivalent to $\alpha\in A'$ and $\beta\in e_\alpha$, which means that $\alpha\in A'^+_{\beta,\vec{e}}$.

Similarly, $\tilde{A'}^-_{\beta,\vec{d}}=A'^-_{\beta,\vec{e}}$:
$\alpha\in\tilde{A'}^-_{\beta,\vec{d}}$ iff $\alpha\in A'$ and $\beta\notin d_\alpha$. As before, $\alpha\ge\tau$, so $d_\alpha=e_\alpha$. Moreover, $\beta<\tau\le\alpha$. Thus, this is equivalent to $\alpha\in A'$ and $\beta\in\alpha\ohne e_\alpha$, which means that $\alpha\in A'^-_{\beta,\vec{e}}$, as wished.

So, we have that $A'^+_{\beta,\vec{e}},A'^-_{\beta,\vec{e}}\in\calB$. Since $A'\sub A$, we have that $A'^+_{\beta,\vec{e}}\sub A^+_{\beta,\vec{e}}$ and $A'^-_{\beta,\vec{e}}\sub A^-_{\beta,\vec{e}}$, so, since $\calB$ is closed under supersets, $A^+_{\beta,\vec{e}},A^-_{\beta,\vec{e}}\in\calB$. Thus, $\vec{e}$ is a $\kappa$-list witnessing that $\SP{\kappa,\tau}{\calA,\calB}$ holds.
\end{proof}

In most of the concrete cases when we will consider $\SP{\kappa,\tau}{\calA,\calB}$ with $\tau\le\kappa$, the assumptions of Observations \ref{observation:IndependentOfInitialSegmentsGivesSplitEquivalentToGeneralSplit} and \ref{obs:ReductionToKappaLists} will be satisfied, so we may think of a $\SP{\kappa,\tau}{\calA,\calB}$-sequence as a $\kappa$-list.

The following lemma says that split principles can be viewed as statements about the sizes of the corresponding splitting numbers.

\begin{lemma}
\label{lemma:SNleLambdaEquivalentToKappaLambdaSplit-General}
Let $\kappa$ and $\tau$ be cardinals, and let $\calA$, $\calB$ be families of subsets of $\kappa$. Then $\SP{\kappa,\tau}{\calA,\calB}$ holds iff $\sn_{\calA,\calB}(\kappa)\le\tau$, that is, iff there is an $(\calA,\calB)$-splitting family of cardinality at most $\tau$.
\end{lemma}

\noindent\emph{Note:} In other words, $\sn_{\calA,\calB}(\kappa)$ is the least $\tau$ such that $\SP{\kappa,\tau}{\calA,\calB}$ holds.

\begin{proof} For the direction from right to left, if $\calS=\{x_\alpha\st\alpha<\tau\}$ is an $(\calA,\calB)$-splitting family for $\kappa$, then we can define a sequence $\kla{d_\alpha\st\alpha<\kappa}$ of subsets of $\tau$ by setting
$d_\alpha=\{\gamma<\tau\st\alpha\in x_\gamma\}.$
Then $\vec{d}$ is a $\SP{\kappa,\tau}{\calA,\calB}$ sequence, because if $A\in\calA$, then there is a $\beta<\tau$ such that both $A\cap x_\beta$ and $A\ohne x_\beta$ belong to $\calB$, but $A\cap x_\beta=\tA^+_{\beta,\vec{d}} \ $ and $A\ohne x_\beta=\tA^-_{\beta,\vec{d}} \ $ so we are done.

Conversely, if $\vec{d}$ is a $\SP{\kappa,\tau}{\calA,\calB}$ sequence, then for each $\gamma<\tau$, we define a subset $x_\gamma$ of $\kappa$ by
$x_\gamma=\{\alpha<\kappa\st\gamma\in d_\alpha\},$ that is, $x_\gamma=\kappa^+_{\gamma,\vec{d}}$.
Then $\calS=\{x_\gamma\st\gamma<\tau\}$ is an $(\calA,\calB)$-splitting family for $\kappa$, because if $A\in\calA$, then there is a $\beta<\tau$ such that both $\tA^+_{\beta,\vec{d}}$ and $\tA^-_{\beta,\vec{d}}$ are in $\calB$, but as before, $\tA^+_{\beta,\vec{d}}=A\cap x_\beta$ and $\tA^-_{\beta, \vec d}=A\ohne x_\beta$.
\end{proof}

What this proof shows is that if $X\sub\kappa\times\lambda$ is a set and we let $X^\beta=\{\alpha<\kappa\st\kla{\alpha,\beta}\in X\}$ be the horizontal section at height $\beta<\lambda$, and $X_\alpha=\{\beta<\lambda\st\kla{\alpha,\beta}\in X\}$ be the vertical section at $\alpha<\kappa$, then $\seq{X_\alpha}{\alpha<\kappa}$ is a \SP{\kappa,\lambda}{\calA,\calB} sequence iff $\set{X^\beta}{\beta<\lambda}$ is an $(\calA,\calB)$-splitting family.

\begin{corollary}
If $\calB$ is a tail set, then $\SP{\kappa}{\calA,\calB}$ holds iff $\sn_{\calA,\calB}(\kappa)\le\kappa$.
\end{corollary}

\begin{observation}
\label{observation:TrivialInequality}
Let $\kappa$ be a cardinal, and let $\calA,\calA',\calB,\calB'$ be collections of subsets of $\kappa$. If $\calA\sub\calA'$ and $\calB'\sub\calB$, then $\sn_{\calA,\calB}(\kappa)\le\sn_{\calA',\calB'}(\kappa)$.
\end{observation}

\begin{proof}
Under the assumptions stated, every $(\calA',\calB')$-splitting family is also $(\calA,\calB)$-splitting.
\end{proof}

\section{Characterizations of large cardinals}
\label{sec:LargeCardinals}

We will now provide our characterizations of large cardinal properties using split principles. We take some detours that give more information about splitting numbers. We begin with inaccessible cardinals.

\subsection{Inaccessible and Regular Cardinals}
\label{subsec:inacc}
Towards characterizing inaccessibility by split principles, we first prove two useful, general lemmas.

\begin{lemma}
\label{lem:AbstractPowerLifting}
Let $\kappa$, $\tau$ be cardinals, let $\mathcal{I}$ be a $\tau^+$-complete ideal on $\kappa$ containing all the singletons. If $2^\tau\ge\kappa$, then $\SP{\kappa,\tau}{\mathcal{I}^+}$ holds.
\end{lemma}

\begin{proof}
Let $\calA=\mathcal{I}^+$, and
let $\seq{d_\alpha}{\alpha<\kappa}$ be such that for $\alpha<\beta<\kappa$, $d_\alpha \neq d_\beta$ are distinct subsets of $\tau$. We claim that $\vec{d}$ is a $\SP{\kappa}{\calA}$-sequence. Assume that it is not, and let $A\in\calA$ be such that no $\beta<\tau$ splits $A$ into $\mathcal{A}$ with respect to $\vec{d}$. Then for every such $\beta$, exactly one of $\tilde{A}^+_{\beta,\vec{d}}$ and $\tilde{A}^-_{\beta,\vec{d}}$ is in $\mathcal{I}$: since $A$ can't be split, it can't be that both are in $\mathcal{A}$, but it also can't be that both are in $\mathcal{I}$, or else $A=\tilde{A}^+_{\beta,\vec{d}}\cup\tilde{A}^+_{\beta,\vec{d}}\in\mathcal{I}$. Let $N_\beta$ be the one of the two sets that is in $\mathcal{I}$.
Then $N=\bigcup_{\beta<\tau}N_\beta\in\mathcal{I}$ since $\mathcal{I}$ is $\tau^+$-closed. Clearly, $A$ is not a subset of $N$, for otherwise, $A$ would lie in $I$. So fix any $\alpha\in A\ohne N$.

We claim that for $\beta<\tau$, $\beta\in d_\alpha$ iff $\tilde{A}^+_{\beta,\vec{d}}$ is in $\calA$. For the direction from left to right, if $\beta\in d_\alpha$, then $\alpha\in\tilde{A}^+_{\beta,\vec{d}}$. So if
$\tilde{A}^+_{\beta,\vec{d}}$ were not in $\calA$, then we'd have that $N_\beta=\tilde{A}^+{\beta,\vec{d}}\in\mathcal{I}$, so $\alpha\in N_\beta$, which is a contradiction, as $\alpha\in A\ohne N$. Vice versa, if $\tilde{A}^+_{\beta,\vec{d}}\in\calA$, then $N_\beta=\tilde{A}^-_{\beta,\vec{d}}$. So since $\alpha\in A\ohne N$, $\alpha\notin\tilde{A}^-_{\beta,\vec{d}}$, so $\alpha\in\tilde{A}^+_{\beta,\vec{d}}$, and this implies that $\beta\in d_\alpha$.

So we have that
\[d_\alpha=\{\beta<\tau\st\tilde{A}^+_{\beta,\vec{d}}\in\calA\}.\]
But note that that the right hand side of this equality is independent of $\alpha$. Thus, for $\alpha_0,\alpha_1\in A\ohne N$ with $\alpha_0\neq\alpha_1$ (such $\alpha_0$, $\alpha_1$ exist by our assumption that $\mathcal{I}$ contains all singletons), we have that $d_{\alpha_0}=d_{\alpha_1}$, a contradiction.
\end{proof}

\begin{definition}
\label{def:RangeWidthSpread}
Let $\vec{d}$ be a $\SP{\kappa,\tau}{\calA,\calB}$-sequence $\vec{d}$. We write $\ran(\vec{d})=\{d_\alpha\st\alpha<\kappa\}$, and define the \emph{width} of $\vec{d}$ to be $\card{\ran(\vec{d})}$. For $d\in\ran(\vec{d})$, let us write $\spread(d)=\{\alpha<\kappa\st d_\alpha=d\}$.
\end{definition}

Note that $\spread(d)$ depends on the sequence $\vec{d}$, but it will usually be clear from the context which sequence is meant.

\begin{observation}
\label{obs:SpreadsArentSplit}
Let $\vec{d}$ be a $\SP{\kappa,\tau}{\calA,\calB}$-sequence, where $\leer\notin\calB$. Then for all $d\in\ran(\vec{d})$, $\spread(d)\notin\calA$.
\end{observation}

\begin{proof}
Otherwise, $\spread(d)$ would have to be split into nonempty sets, but $\vec{d}$ is constant on $\spread(d)$, so this is impossible.
\end{proof}

\begin{observation}
\label{obs:Width>=Completeness}
Let $\calI$ be a $\bkappa$-complete proper ideal on the cardinal $\kappa$, and let $\vec{d}$ be a $\SP{\kappa,\tau}{\calI^+,\calB}$-sequence, where
$\leer\notin\calB$. Then the width of $\vec{d}$ must be at least $\bkappa$.
\end{observation}

\begin{proof}
Otherwise $\{\spread(d)\st d\in\ran(\vec{d})\}$ is a partition of $\kappa$ into fewer than $\bkappa$ pieces, and by Observation \ref{obs:SpreadsArentSplit}, each piece is in $\calI$. This contradicts that $\calI$ is $\bkappa$-complete and proper.
\end{proof}

We can now quickly derive a kind of converse to Lemma \ref{lem:AbstractPowerLifting}.

\begin{lemma}
\label{lem:AbstractPowerConverse}
Let $\tau$, $\kappa$ be a cardinals, and let $\mathcal{I}$ be a $(2^\tau)^+$-complete proper ideal on $\kappa$. Then $\SP{\kappa,\tau}{\mathcal{I}^+,\Nonempty}$ fails.
\end{lemma}

\begin{proof}
Suppose $\vec{d}=\seq{d_\alpha}{\alpha<\kappa}$ were a $\SP{\kappa,\tau}{\calA,\Nonempty}$-sequence. Being a sequence of subsets of $\tau$, its width can be at most $2^\tau$. But $\calI$ is $(2^\tau)^+$-complete, so by Observation \ref{obs:Width>=Completeness}, the width of $\vec{d}$ must be at least $(2^\tau)^+$.
\end{proof}


\begin{theorem}
\label{thm:AbstractCharacterizationOfInaccessibility}
Let $\kappa$ be an uncountable regular cardinal, and let $\mathcal{I}$ be a $\kappa$-complete proper ideal on $\kappa$ containing all singletons. Then the following are equivalent:
\begin{enumerate}[label={\textnormal{(\arabic*)}}]
  \item
  \label{item:KappaIsInaccessible}
  $\kappa$ is inaccessible.
  \item
  \label{item:NotSPkappatau(I^+,nonempty)}
  For every cardinal $\tau<\kappa$, $\SP{\kappa,\tau}{\mathcal{I}^+,\Nonempty}$ fails.
  \item
  \label{item:NotSPkappatau(I^+)}
  For every cardinal $\tau<\kappa$, $\SP{\kappa,\tau}{\mathcal{I}^+}$ fails.
\end{enumerate}
\end{theorem}

\begin{proof}
\ref{item:KappaIsInaccessible}$\implies$\ref{item:NotSPkappatau(I^+,nonempty)}: Assume $\kappa$ is inaccessible, and let $\tau<\kappa$ be a cardinal. Then $2^\tau<\kappa$, so $\mathcal{I}$ is $(2^\tau)^+$-complete. So Lemma \ref{lem:AbstractPowerConverse} applies, showing that $\SP{\kappa,\tau}{\mathcal{I}^+,\Nonempty}$ fails.

\ref{item:NotSPkappatau(I^+,nonempty)}$\implies$\ref{item:NotSPkappatau(I^+)}: This is trivial.

\ref{item:NotSPkappatau(I^+)}$\implies$\ref{item:KappaIsInaccessible}: Let us show the contrapositive. So we assume that $\kappa$ is not inaccessible. Since $\kappa$ is regular, there is a $\tau<\kappa$ such that $2^\tau\ge\kappa$. The ideal $\mathcal{I}$ is then $\tau^+$-complete, so that Lemma \ref{lem:AbstractPowerLifting} applies, showing that $\SP{\kappa,\tau}{\mathcal{I}^+}$ holds.
\end{proof}

The following lemma characterizes inaccessibility by concrete instances of the previous theorem. The equivalence \ref{item:Inaccessible}$\iff$\ref{item:Unbounded} follows from work of Motoyoshi \cite{Motoyoshi:Masters}.

\begin{lemma}
\label{lemma:CharacterizationOfInaccessibility}
Let $\kappa$ be an uncountable regular cardinal. The following are equivalent:
\begin{enumerate}[label=\textnormal{(\arabic*)}]
\item
\label{item:Inaccessible}
$\kappa$ is inaccessible.
\item
\label{item:StationaryNonempty}
$\SP{\kappa,\tau}{\Stationary,\Nonempty}$ fails for every $\tau<\kappa$.
\\
Equivalently, $\sn_{\Stationary,\Nonempty}(\kappa)\ge\kappa$.
\item
\label{item:Stationary}
$\SP{\kappa,\tau}{\Stationary}$ fails for every $\tau<\kappa$.
\\
Equivalently, $\sn_{\Stationary}(\kappa)\ge\kappa$.
\item
\label{item:UnboundedNonempty}
$\SP{\kappa,\tau}{\Unbounded,\Nonempty}$ fails for every $\tau<\kappa$.
\\
Equivalently, $\sn_{\Unbounded,\Nonempty}(\kappa)\ge\kappa$.
\item
\label{item:Unbounded}
$\SP{\kappa,\tau}{\Unbounded}$ fails for every $\tau<\kappa$.
\\
Equivalently, $\sn_{\Unbounded}(\kappa)\ge\kappa$.
\end{enumerate}
It follows that for any collection $\calB$ with $\Stationary\sub\calB\sub\Nonempty$, these conditions are equivalent to the failure of $\SP{\kappa,\tau}{\Stationary,\calB}$, and similarly for any $\calB$ with $\Unbounded\sub\calB\sub\Nonempty$, they are equivalent to the failure of $\SP{\kappa,\tau}{\Unbounded,\calB}$.

Moreover, each of \ref{item:StationaryNonempty}, \ref{item:UnboundedNonempty} are equivalent to $\kappa$ being inaccessible even if $\kappa$ is not assumed to be regular.
\end{lemma}

\begin{proof} The equivalence \ref{item:Inaccessible}$\iff$\ref{item:Unbounded} derives from previously known results as follows. According to \cite{Suzuki:OnSplittingNumbers}, it was shown in \cite{Motoyoshi:Masters} that for an uncountable regular cardinal $\kappa$, $\kappa$ is inaccessible iff $\sn(\kappa)\ge\kappa$ (see \cite[Lemma 3]{Zapletal:SplittingNumberAtUncountableCards} for a proof). By Lemma \ref{lemma:SNleLambdaEquivalentToKappaLambdaSplit-General}, this is equivalent to saying that for no $\tau<\kappa$ does $\SP{\kappa,\tau}{\Unbounded}$ hold.

But of course, all of the claimed equivalences are just instances of the previous theorem, using nonstationary ideal on $\kappa$ and the ideal of the bounded subsets of $\kappa$, both of which are $\kappa$-complete, since $\kappa$ is regular.

The claim about families $\calB$ with $\Stationary\sub\calB\sub\Nonempty$ follows from the equivalence of \ref{item:StationaryNonempty}~and \ref{item:Stationary}, and the claim about $\calB$ with $\Unbounded\sub\calB\sub\Nonempty$ follows from the equivalence of \ref{item:UnboundedNonempty} and \ref{item:Unbounded}.

For the last claim, it suffices to show that \ref{item:UnboundedNonempty} implies that $\kappa$ is regular. But this is obvious, since if $\seq{\xi_\alpha}{\alpha<\cf(\kappa)}$ is cofinal in $\kappa$, then $\{\xi_\alpha\st\alpha<\cf(\kappa)\}$ (viewed as a collection of subsets of $\kappa$) is an $(\Unbounded,\Nonempty)$-splitting family. So $\sn_{\Unbounded,\Nonempty}(\kappa)\le\cf(\kappa)$, so $\SP{\kappa,\cf(\kappa)}{\Unbounded,\Nonempty}$ holds. So it has to be the case that $\cf(\kappa)=\kappa$.
\end{proof}

The two lemmas from the beginning of this section give us some more information even if $\kappa$ is not inaccessible.

\begin{corollary}
\label{cor:BothWaysForRegulars}
Let $\tau<\kappa$ be cardinals, where $\kappa$ is regular. Let $\mathcal{I}\sub\power{\kappa}$ be a $\kappa$-complete proper ideal containing all the singletons. Then, letting $\calA=\mathcal{I}^+$,
$\SP{\kappa,\tau}{\calA}$ holds iff $2^\tau\ge\kappa$.
\end{corollary}

\noindent\emph{Note:} This means that in the setting of the corollary, $\mathfrak{s}_\calA(\kappa)<\kappa$ iff there is a $\tau<\kappa$ with $2^\tau\ge\kappa$, and in this situation, $\mathfrak{s}_\calA(\kappa)$ is the least $\tau<\kappa$ such that $2^\tau\ge\kappa$. The corollary applies when $\calA$ is $\Unbounded_\kappa$, $[\kappa]^\kappa$, or $\Stationary_\kappa$ (if $\cf(\kappa)>\omega$), for example.

\begin{proof}
If $2^\tau\ge\kappa$, then Lemma \ref{lem:AbstractPowerLifting} applies, since $I$ is $\tau^+$-complete, giving us that $\SP{\kappa,\tau}{\calA}$ holds.

On the other hand, if $2^\tau<\kappa$, then Lemma \ref{lem:AbstractPowerConverse} applies, because $I$ is $(2^\tau)^+$-complete. So $\SP{\kappa,\tau}{\calA,\Nonempty}$ fails, hence $\SP{\kappa,\tau}{\calA}$ fails as well, as $\calA\sub\Nonempty$.
\end{proof}

Since it fits into this context, let us point out that the regularity of a limit ordinal $\kappa$ can be expressed using a thin version of the two cardinal split principle.

\begin{definition}
\label{def:TwoCardinalSplit}
For cardinals $\kappa$ and $\tau$, and sets $\calA,\calB\sub\power{\kappa}$, a \emph{thin} $\SP{\kappa,\tau}{\calA,\calB}$-sequence is a $\SP{\kappa,\tau}{\calA,\calB}$-sequence with the property that for each $\beta<\tau$, the set $\{d_\alpha\cap\beta\st\alpha<\kappa\}$ has cardinality less than $\kappa$. The principle thin $\SP{\kappa,\tau}{\calA,\calB}$ says that there is a thin $\SP{\kappa,\tau}{\calA,\calB}$-sequence.
\end{definition}

If $\tau<\kappa$, $\calA$ is a tail set and $\calB$ is closed under supersets, the principle \emph{thin} $\SP{\kappa,\tau}{\calA,\calB}$ can be witnessed by a thin $\kappa$-list, by Observation \ref{obs:ReductionToKappaLists}, in which case we are in line with Definition \ref{definition:ThinSplit}.

\begin{observation}
\label{obs:CharacterizationOfRegularity}
Let $\kappa$ be a limit ordinal. The following are equivalent:
\begin{enumerate}[label=\textnormal{(\arabic*)}]
\item
\label{item:Regular}
$\kappa$ is regular.
\item
\label{item:ThinUnboundedNonempty}
Thin $\SP{\kappa,\tau}{\Unbounded,\Nonempty}$ fails for every $\tau<\kappa$.
\end{enumerate}
\end{observation}

\begin{proof}
\ref{item:Regular}$\implies$\ref{item:ThinUnboundedNonempty}: Let $\kappa$ be regular, and suppose thin $\SP{\kappa,\tau}{\Unbounded,\Nonempty}$ held, for some $\tau<\kappa$. Let $\vec{d}$ be a thin $\kappa$-list witnessing this. Note that the set $\{d_\alpha\st\alpha<\kappa\}=\{d_\alpha\cap\tau\st\alpha<\kappa\}$ has cardinality less than $\kappa$. Hence, there is a stationary set $S\sub\kappa$ (with $S\cap\tau=\leer$) on which $\vec{d}$ is constant. Clearly then, for every $\beta<\tau$, one of $S^+_\beta$ and $S^-_\beta$ is empty, a contradiction.


\ref{item:ThinUnboundedNonempty}$\implies$\ref{item:Regular}: Assume thin $\SP{\kappa,\tau}{\Unbounded,\Nonempty}$ fails for every $\tau<\kappa$, but suppose $\kappa$ is singular. Let $\tau=\cf(\kappa)<\kappa$. Let $\seq{\xi_\gamma}{\gamma<\tau}$ be monotone and cofinal in $\kappa$. As was pointed out in the last part of the proof of Lemma \ref{lemma:CharacterizationOfInaccessibility}, $\{\xi_\gamma\st\gamma<\tau\}$ is an $(\Unbounded,\Nonempty)$-splitting family. So, as before, if we define,
for $\alpha<\kappa$:
\[d_\alpha=\{\gamma<\tau\st\alpha<\xi_\gamma\}\]
then $\vec{d}$ is an $\SP{\kappa,\tau}{\calA,\calB}$ sequence. But note that $d_\alpha$ is of the form $\tau\ohne\delta_\alpha$, where $\delta_\alpha<\tau$ is the least $\gamma$ such that $\alpha<\xi_\gamma$. So $\{d_\alpha\st\alpha<\kappa\}$ has cardinality at most $\tau$. Moreover, if we define $e_\alpha=d_\alpha\cap\alpha$, then $\vec{e}$ is a $\kappa$-list that also witnesses $\SP{\kappa,\tau}{\calA,\calB}$, by Observation \ref{obs:ReductionToKappaLists}, and $\vec{e}$ is thin. This contradicts our assumption.
\end{proof}

\subsection{Weakly compact cardinals}
\label{subsec:wc}

We show that the original split principle can be used to characterize weakly compact cardinals.

\begin{corollary}
\label{corollary:SplitKappaIffNotWC}
Let $\kappa$ be a regular uncountable cardinal. The following are equivalent:
\begin{enumerate}[label=\textnormal{(\arabic*)}]
	\item \label{item:wc} $\kappa$ is weakly compact
	\item \label{item:OGSplit} $\SP{\kappa}{\Unbounded}$ fails.
\end{enumerate}	
\end{corollary}

\begin{proof}
We shall show both directions of the equivalence separately.

\ref{item:wc}$\implies$\ref{item:OGSplit}: We prove the contrapositive, so assuming $\SP{\kappa}{\Unbounded}$, we have to show that $\kappa$ is not weakly compact. If $\kappa$ is not inaccessible, then we are done, so let's assume it is. By Corollary \ref{corollary:TreeCharacterizationIfKappaIsRegular}, there is a sequential tree $T\sub{}^{{<}\kappa}2$ of height $\kappa$ with no cofinal branch. Since $\kappa$ is inaccessible, $T$ is a $\kappa$-tree, and thus, $\kappa$ does not have the tree property, so $\kappa$ is not weakly compact.

\ref{item:OGSplit}$\implies$\ref{item:wc}: Let $\kappa$ fail to be weakly compact. We split into two cases.

\emph{Case 1: $\kappa$ is not inaccessible.} Then by Lemma \ref{lemma:CharacterizationOfInaccessibility}, there is a $\tau<\kappa$ such that $\SP{\kappa,\tau}{\Unbounded}$ holds. This clearly implies that $\SP{\kappa,\kappa}{\Unbounded}$ holds. Moreover, since $\Unbounded$ is a tail set, by Observation \ref{observation:IndependentOfInitialSegmentsGivesSplitEquivalentToGeneralSplit}, $\SP{\kappa,\kappa}{\Unbounded}$ is equivalent to $\SP{\kappa}{\Unbounded}$.

\emph{Case 2: $\kappa$ is inaccessible.} Since $\kappa$ is not weakly compact, then the tree property fails at $\kappa$, and this is witnessed by a sequential tree $T$ on ${}^{{<}\kappa}2$ that has no cofinal branch.
Thus, $\Split{\kappa}$ holds, by Corollary \ref{corollary:TreeCharacterizationIfKappaIsRegular}.
\end{proof}

Note that Corollary \ref{corollary:SplitKappaIffNotWC} would hold for $\kappa=\omega$ as well, if we considered $\omega$ to be weakly compact, which is not the standard usage.

So an uncountable regular cardinal $\kappa$ is weakly compact iff $\SP{\kappa}{\Unbounded}$ fails. Since this is equivalent to saying that $\SP{\kappa,\kappa}{\Unbounded}$ fails, it can be equivalently expressed by saying that $\sn(\kappa)>\kappa$, by Lemma \ref{lemma:SNleLambdaEquivalentToKappaLambdaSplit-General}. This latter characterization of weak compactness was shown in \cite{Suzuki:OnSplittingNumbers}.

%
%
%

\subsection{Ineffable cardinals and coherence}
\label{subsec:Ineffables}

Ineffable cardinals were introduced by Jensen and Kunen \cite{Jensen:CombinatorialProperties} and have become a staple of the theory of large cardinals. A regular cardinal $\kappa$ is ineffable if whenever $\seq{d_\alpha}{\alpha<\kappa}$ is a $\kappa$-list, there is a stationary set $S\sub\kappa$ on which $\vec{d}$ coheres, meaning that for all $\alpha<\beta$, both in $S$, $d_\alpha=d_\beta\cap\alpha$. Since $\kappa$-lists figure so prominently in this large cardinal property, it is natural to expect a tight connection to split principles.

\begin{definition}
\label{def:IneffableBranchesEtc}
Let $\kappa$ be a cardinal, and let $\calA$ be a family of subsets of $\kappa$. A subset $b\sub\kappa$ is an $\calA$ branch of the $\kappa$-list $\vec d = \seq{ d_\alpha }{\alpha< \kappa}$ iff there is a set $A\in\calA$ such that for all $\alpha\in A$, we have that
$d_\alpha = b\cap\alpha$. A $\Stationary$ branch is called an \emph{ineffable} branch, and an $\Unbounded$ branch is called an \emph{almost ineffable} branch. $\kappa$ is $\calA$-ineffable if $\kappa$ is regular and every $\kappa$-list has an $\calA$ branch.
\end{definition}

Note that $\vec{d}$ has an $\calA$ branch iff there is an $A\in\calA$ on which $\vec{d}$ coheres (in which case $\bigcup_{\alpha\in A}d_\alpha$ is an $\calA$ branch). So using this language, regular, \Stationary-ineffable cardinals are exactly \textit{ineffable cardinals}, and regular, \Unbounded-ineffable ones are exactly \textit{almost ineffable cardinals}; see \cite{Jensen:CombinatorialProperties}, where these were introduced.

We start by observing that splitting into nonempty sets is equivalent to the failure of coherence.

\begin{lemma}
\label{lemma:Coherence=NotSplittingIntoNonempty}
Let $\kappa$ be a cardinal, $A\sub\kappa$, and let $\vec{d}$ be a $\kappa$-list. Then the following are equivalent:
\begin{enumerate}[label=\textnormal{(\arabic*)}]
  \item $\vec{d}$ splits $A$ into nonempty sets (i.e., there is a $\beta$ such that both $A^+_\beta$ and $A^-_\beta$ are nonempty).
  \item $\vec{d}$ does not cohere on $A$.
\end{enumerate}
\end{lemma}

\begin{proof}
(1)$\implies$(2): Let $\beta$ be as in (1). Let $\gamma\in A^+_\beta$ and $\delta\in A^-_\beta$. Then $\beta\in d_\gamma$ and $\beta\notin d_\delta$, but $\beta<\delta$. Hence, $\vec{d}$ does not cohere on $\{\gamma,\delta\}$.

(2)$\implies$(1): We show the contrapositive. So assuming $\vec{d}$ does not split $A$ into nonempty sets, we show that $\vec{d}$ coheres on $A$. To see this, let $\gamma<\delta$ be given, both in $A$. Then for every $\beta<\gamma$, $\beta\in d_\gamma$ iff $\beta\in d_\delta$, for if not, say if $\beta\in d_\gamma$ but $\beta\notin d_\delta$, then $\gamma\in A^+_{\beta,\vec{d}}$ and $\delta\in A^-_{\beta,\vec{d}}$ (and if $\beta\notin d_\gamma$ but $\beta\in d_\delta$, then it's the other way around). In any case, we'd have a $\beta<\gamma$ such that both $A^+_{\beta,\vec{d}}$ and $A^-_{\beta,\vec{d}}$ are nonempty, a contradiction. Thus, $d_\gamma=d_\delta\cap\gamma$, as wished.
\end{proof}

This results in a general and uniform characterization of $\calA$-ineffability in terms of split principles.

\begin{theorem}
\label{theorem:GeneralABranchesForKappaListsAndGeneralIneffability}
Let $\kappa$ be a cardinal, $\calA$ a family of subsets of $\kappa$ and $\vec{d}$ a $\kappa$-list. Then the following are equivalent:
\begin{enumerate}[label=\textnormal{(\arabic*)}]
  \item $\vec{d}$ is a \SP{\kappa}{\calA,\Nonempty}-sequence.
  \item $\vec{d}$ has no $\calA$ branch.
\end{enumerate}
Thus, a regular, uncountable cardinal $\kappa$ is $\calA$-ineffable if and only if $\SP{\kappa}{\calA,\Nonempty}$ fails.
\end{theorem}

\begin{proof}
(1)$\implies$(2): Suppose $B$ is an $\calA$ branch for $\vec{d}$. Let $A\in\calA$ be such that for all $\alpha\in A$, $d_\alpha=B\cap\alpha$. Then $\vec{d}$ coheres on $A$. So by Lemma \ref{lemma:Coherence=NotSplittingIntoNonempty}, $\vec{d}$ does not split $A$ into nonempty sets, a contradiction.

(2)$\implies$(1): If $\vec{d}$ has no $\calA$ branch, then $\vec{d}$ coheres on no $A\in\calA$, so by Lemma \ref{lemma:Coherence=NotSplittingIntoNonempty}, $\vec{d}$ splits every $A\in\calA$ into nonempty sets.
\end{proof}

A comment about this theorem is in order. The attentive reader will have noticed that in its proof, we used the particular way $A^-_\beta$ was defined: $\{\alpha\in A\st\beta\in\alpha\ohne d_\alpha\}$ instead of $\{\alpha\in A\st\beta\notin d_\alpha\}$, which is how $\tilde{A}^-_\beta$ was defined. In the context of $\kappa$-lists, it makes intuitive sense to define $A^-_\beta$ the way we did, because after all $d_\alpha$ can only give positive information up to $\alpha$ (since $d_\alpha\sub\alpha$), so why should it be allowed to give unbounded negative information? Namely, let us consider the very simple $\kappa$-list $d_\alpha=\alpha$ for $\alpha<\kappa$. We claim that it is a $\SP{\kappa,\kappa}{[\kappa]^{{\ge}2},\Nonempty}$-sequence (this is the strongest possible split principle splitting into $\Nonempty$); recall that the two cardinal split principle was formulated with $\tilde{A}^+_\beta$ and $\tilde{A}^-_\beta$, so the ``asymmetric'' version. In other words, we are claiming that if we hadn't used $A^-_\beta$ when defining $\SP{\kappa}{\calA,\Nonempty}$, the principle would trivialize. To see this, let $A\sub\kappa$ be arbitrary, and fix $\beta<\kappa$. Then \[\tilde{A}^+_\beta=A^+_\beta=\{\alpha\in A\st\beta\in d_\alpha\}=\{\alpha\in A\st\beta\in\alpha\}=A\ohne(\beta+1),\]
and
\[\tilde{A}^-_\beta=\{\alpha\in A\st\beta\notin d_\alpha\}=\{\alpha\in A\st\beta\notin\alpha\}=A\cap(\beta+1).\]
Thus, if $A$ has at least two elements, we can let $\beta$ be its minimum and both $\tilde{A}^+_\beta$ and $\tilde{A}^-_\beta$ turn out nonempty, proving the principle. If we had worked with the symmetric version, we would have:
\[ A^-_\beta=\{\alpha\in A\st\beta\in\alpha\ohne d_\alpha\}=\{\alpha\in A\st\beta\in\alpha\ohne\alpha\}=\leer,\]
so this particular sequence $\vec{d}$ does not split any set into two nonempty sets in the sense of $\SP{\kappa}{\calA,\Nonempty}$. In fact, we have shown that if $\kappa$ is regular then $\SP{\kappa}{\calA,\Nonempty}$ only holds if $\kappa$ is $\calA$-ineffable.

Observation \ref{observation:IndependentOfInitialSegmentsGivesSplitEquivalentToGeneralSplit} does not apply here, because $\calB=\Nonempty$ is not a tail set. Thus, even though a special case of Theorem \ref{theorem:GeneralABranchesForKappaListsAndGeneralIneffability} shows that an uncountable regular $\kappa$ is ineffable if and only if $\SP{\kappa}{\Stationary,\Nonempty}$ fails, this does not provide us with a characterization of ineffability in terms of splitting numbers and splitting families. This is why the following improvement of Lemma \ref{lemma:Coherence=NotSplittingIntoNonempty} for the context of splitting stationary sets is useful.

\begin{lemma}
\label{lem:CoherencyOnStationarySet=NotSplittingStationaryIntoStationary-Abstract}
\label{lem:CoherencyOnStationarySet=NotSplittingStationaryIntoStationary}
Let $\kappa$ be an uncountable regular cardinal, and let $\mathcal{I}$ be a normal ideal on $\kappa$ containing all the singletons. Then the following are equivalent for a $\kappa$-list $\vec{d}$:
\begin{enumerate}[label={\textnormal{(\arabic*)}}]
  \item
  \label{item:ListCoheresOnI^+Set}
  $\vec{d}$ coheres on a set in $\mathcal{I}^+$.
  \item
  \label{item:ListDoesNotSplitI^+}
  $\vec{d}$ is not a $\SP{\kappa}{\mathcal{I}^+}$-sequence.
\end{enumerate}
In particular, this is true when $\mathcal{I}$ is the nonstationary ideal on $\kappa$, so that $\vec{d}$ coheres on a stationary subset of $\kappa$ iff $\vec{d}$ is not a $\SP{\kappa}{\Stationary}$-sequence.
\end{lemma}

\begin{proof}
\ref{item:ListCoheresOnI^+Set}$\implies$\ref{item:ListDoesNotSplitI^+}: If $\vec{d}$ coheres on a set $S\in\mathcal{I}^+$, say, then $\vec{d}$ does not split $S$ into nonempty sets, by Lemma \ref{lemma:Coherence=NotSplittingIntoNonempty}, and in particular, it does not split $S$ into sets from $\mathcal{I}^+$, hence \ref{item:ListDoesNotSplitI^+} holds.

\ref{item:ListDoesNotSplitI^+}$\implies$\ref{item:ListCoheresOnI^+Set}:
Since $\vec{d}$ is assumed not to be a $\SP{\kappa}{\mathcal{I}^+}$-sequence, it follows that there is set $S\in\mathcal{I}^+$ such that for no $\beta<\kappa$ do we have that both $S^+_\beta=S^+_{\beta,\vec{d}}$ and $S^-_\beta=S^-_{\beta,\vec{d}}$ are in $\mathcal{I}^+$. Since $S=(S\cap(\beta+1))\cup S^+_\beta\cup S^-_\beta$ and $\beta+1\in\mathcal{I}$, it follows that exactly one of $S^+_\beta$ and $S^-_\beta$ is in $\mathcal{I}^+$ and the other one is not. For each $\beta< \kappa$ let $N_\beta$ be the one that is in $\mathcal{I}$. Let $N=\dunion\limits_{\beta<\kappa}N_\beta$. So $N\in\mathcal{I}$ by normality, and hence, $T=S\ohne N\in\mathcal{I}^+$.
Moreover, $\vec{d}$ coheres on $T$, because if $\gamma<\delta$ both are members of $T$, then for $\beta<\gamma$, it follows that $\gamma,\delta\notin N_\beta$. So if $N_\beta=S^-_\beta\in\mathcal{I}$, then $\gamma,\delta\in S^+_\beta$, which means that $\beta\in d_\gamma$ and $\beta\in d_\delta$. And if $N_\beta=S^+_\beta\in\mathcal{I}$, then $\gamma,\delta\in S^-_\beta$ and it follows that $\beta\notin d_\gamma$ and $\beta\notin d_\delta$. So $d_\gamma=d_\delta\cap\gamma$.
\end{proof}

Hence, the following theorem provides a characterization of ineffability in terms of splitting numbers. After completing the work on this paper, with the generous help of Saka\'{e} Fuchino, we were able to obtain a copy of the unpublished preprint \cite{Kamo:SplittingNumbers} by Shizuo Kamo, and learned that the equivalence between (1) and the statement about the splitting number $\sn_{\Stationary}(\kappa)$ in (4) of the following theorem was known to him.

\begin{theorem}
\label{theorem:CharacterizationOfIneffability}
Let $\kappa$ be a regular cardinal. The following are equivalent.
\begin{enumerate}[label=\textnormal{(\arabic*)}]
\item $\kappa$ is ineffable.
\item $\SP{\kappa}{\Stationary,\Nonempty}$ fails.
\item $\SP{\kappa}{\Stationary,\Unbounded}$ fails. Equivalently, $\sn_{\Stationary,\Unbounded}(\kappa)>\kappa$.
\item $\SP{\kappa}{\Stationary}$ fails. Equivalently, $\sn_\Stationary(\kappa)>\kappa$.
\end{enumerate}
It follows that for any family $\calB\sub\power{\kappa}$ satisfying $\Stationary\sub\calB\sub\Nonempty$, the failure of $\SP{\kappa}{\Stationary,\calB}$ characterizes the ineffability of $\kappa$.
\end{theorem}

\begin{proof}
(1)$\iff$(2) follows from Theorem \ref{theorem:GeneralABranchesForKappaListsAndGeneralIneffability}.

(2)$\implies$(3)$\implies$(4) is trivial.

(4)$\implies$(1): Assuming $\SP{\kappa}{\Stationary}$ fails, we have to show that $\kappa$ is ineffable. So let $\vec{d}$ be a $\kappa$-list. By assumption, it does not split stationary sets into stationary sets. So by Lemma \ref{lem:CoherencyOnStationarySet=NotSplittingStationaryIntoStationary}, $\vec{d}$ coheres on a stationary set. By the remark after Definition \ref{def:IneffableBranchesEtc}, $\vec{d}$ has an ineffable branch.

The statements about the splitting numbers follow because the families we are splitting into are tail sets, so that Observation \ref{observation:IndependentOfInitialSegmentsGivesSplitEquivalentToGeneralSplit} applies.
\end{proof}

\begin{corollary}
\label{cor:CharacterizationOfAlmostIneffability}
Let $\kappa$ be a regular cardinal. The following are equivalent.
\begin{enumerate}[label=\textnormal{(\arabic*)}]
\item $\kappa$ is almost ineffable.
\item $\SP{\kappa}{\Unbounded,\Nonempty}$ fails.
\end{enumerate}
\end{corollary}

\begin{proof}
This follows from Theorem \ref{theorem:GeneralABranchesForKappaListsAndGeneralIneffability} using $\calA=\Unbounded_\kappa$.
\end{proof}

Unlike in Theorem \ref{theorem:CharacterizationOfIneffability}, where we showed that for regular $\kappa$, the principles $\SP{\kappa}{\Stationary,\Nonempty}$ and $\SP{\kappa}{\Stationary,\Stationary}$ are equivalent, in fact it is not the case that $\SP{\kappa}{\Unbounded,\Nonempty}$ and $\SP{\kappa}{\Unbounded,\Unbounded}$ are equivalent. Indeed, the former is equivalent to $\kappa$ not being almost ineffable, and the latter is equivalent to $\kappa$ not being weakly compact. In fact, there is no family $\calB$ that is a tail set closed under supersets such that the failure of $\SP{\kappa}{\Unbounded,\calB}$ would characterize $\kappa$'s almost ineffability in any nontrivial way:

If $\kappa$ is not almost ineffable, then we would have to have that $\SP{\kappa}{\Unbounded,\calB}$ holds. Since $\calB$ is supposed to be closed under supersets, it would follow that $\Unbounded\sub\calB$. But $\calB$ shouldn't be $\Unbounded_\kappa$, since that characterizes the failure of weak compactness, and $\kappa$ might be weakly compact. So $\calB$ should contain a bounded subset of $\kappa$, but then since it is a tail set, it would contain the empty set, and then it would have to contain every subset of $\kappa$, being closed under supersets. But then the principle would trivially hold, so its failure couldn't characterize almost ineffability.

Of course, this statement has to be taken with a grain of salt, as letting $\calB=\{A\sub\kappa\st\kappa$ is not almost ineffable$\}$, we trivially have that $\kappa$ is almost ineffable iff $\SP{\kappa}{\Unbounded,\calB}$ fails (hence the disclaimer ``in any nontrivial way.'') Still, Corollary \ref{cor:CharacterizationOfAlmostIneffability} provides a characterization of almost ineffability by the failure of a split principle that does \emph{not} translate to a statement about splitting numbers or splitting families in any meaningful way.

Next, let us make a connection to the ineffable tree property, introduced in \cite[Def.~1.2.7]{Weiss:Diss}; see also \cite{CHMNSU:ITPandSCH}.

\begin{definition}
A regular cardinal $\kappa$ has the \emph{ineffable tree property} if every thin $\kappa$-list has an ineffable branch. This is abbreviated $\mathsf{ITP}(\kappa)$.
\end{definition}

\begin{corollary}
\label{cor:CharacterizationOfITP}
Let $\kappa$ be a regular cardinal. The following are equivalent.
\begin{enumerate}[label=\textnormal{(\arabic*)}]
\item $\mathsf{ITP}(\kappa)$.
\item Thin $\SP{\kappa}{\Stationary,\Nonempty}$ fails.
\item Thin $\SP{\kappa}{\Stationary,\Unbounded}$ fails.
\item Thin $\SP{\kappa}{\Stationary}$ fails.
\end{enumerate}
\end{corollary}

\begin{proof}
This follows immediately from Theorem \ref{theorem:GeneralABranchesForKappaListsAndGeneralIneffability} and Lemma \ref{lem:CoherencyOnStationarySet=NotSplittingStationaryIntoStationary}.
\end{proof}

We briefly digress to point out a connection to square principles. These principles have a rich history, of course, beginning with Jensen \cite{FS}. The form we are going to touch upon is commonly denoted $\Box(\kappa)$; see Schimmerling \cite{Schimmerling:CoherentSequencesAndThreads} and Todorcevic \cite{Todorcevic:WalksBook}, for example.

\begin{definition}
\label{def:square(kappa)}
A $\kappa$-list $\vec{c}=\seq{c_\alpha}{\alpha<\kappa}$ is a \emph{coherent sequence} if for every limit ordinal $\alpha<\kappa$, $c_\alpha$ is a club subset of $\alpha$ and for every $\beta$ which is a limit point of $c_\alpha$, $c_\beta=c_\alpha\cap\beta$. A set $t\sub\kappa$ is a \emph{thread} of $\vec{d}$ if $\vec{c}{{{}^\frown}}t$ is a coherent sequence, i.e., if in the case that $\kappa$ is a limit ordinal, $t$ is club in $\kappa$ and for every limit point $\alpha<\kappa$ of $t$, $c_\alpha=t\cap\alpha$. $\vec{c}$ is a $\Box(\kappa)$-sequence if it is a coherent sequence with no thread.
\end{definition}

\begin{observation}
Let $\kappa$ be a cardinal of uncountable cofinality, and let $\vec{c}$ be a coherent sequence of length $\kappa$. Then the following are equivalent:
\begin{enumerate}[label=\textnormal{(\arabic*)}]
\item
\label{item:Box(kappa)}
$\vec{c}$ is a $\Box(\kappa)$-sequence.
\item
\label{item:SPStatNonempty}
$\vec{c}$ is a $\SP{\kappa}{\Stationary,\Nonempty}$-sequence.
\item
\label{item:SPClubNonempty}
$\vec{c}$ is a $\SP{\kappa}{\mathsf{club},\Nonempty}$-sequence. Here, $\mathsf{club}$ denotes the collection of all club subsets of $\kappa$.
\end{enumerate}
\end{observation}

\begin{proof}
\ref{item:Box(kappa)}$\implies$\ref{item:SPStatNonempty}:
Let $\vec{c}$ be a $\Box(\kappa)$-sequence. If $\vec{c}$ is not a $\SP{\kappa}{\Stationary,\Nonempty}$ sequence, there is a stationary set $S\sub\kappa$ which is not split by $\vec{c}$ into two nonempty sets, which by Lemma \ref{lemma:Coherence=NotSplittingIntoNonempty} means that $\vec{c}$ coheres on $S$. Let $t=\bigcup_{\alpha\in S}c_\alpha$. Then for all $\alpha\in S$, $t\cap\alpha=c_\alpha$. Since unboundedly many $\alpha\in S$ are limit ordinals, as $\cf(\kappa)>\omega$, we have that $t\cap\alpha$ is unboundedly often a club subset of $\alpha$. Hence, $t$ is club in $\kappa$. For the same reason, $t$ is a thread: if $\alpha<\kappa$ is a limit point of $t$, then we can find a $\beta>\alpha$, $\beta\in S$ which is a limit ordinal. Then $t\cap\beta=c_\beta$, so $\alpha$ is a limit point of $c_\beta$, so $t\cap\alpha=c_\beta\cap\alpha=c_\alpha$ since $\vec{c}$ is coherent. So $\vec{c}$ is not a $\Box(\kappa)$-sequence, a contradiction.

\ref{item:SPStatNonempty}$\implies$\ref{item:SPClubNonempty}: This implication is trivial, as every club subset of $\kappa$ is stationary.

\ref{item:SPClubNonempty}$\implies$\ref{item:Box(kappa)}:
Suppose $\vec{c}$ is a coherent $\SP{\kappa}{\mathsf{club},\Nonempty}$-sequence. If it is not a $\Box(\kappa)$-sequence, then it has a thread $t$. But then $\vec{c}$ would cohere on the set $t'$ of limit points of $t$. Thus, by Lemma \ref{lemma:Coherence=NotSplittingIntoNonempty}, $\vec{c}$ would not split $t'$ into nonempty sets. But $t'$ is club in $\kappa$, so $\vec{c}$ would not be a $\SP{\kappa}{\mathsf{club},\Nonempty}$-sequence after all, a contradiction.
\end{proof}

Let us further consider the relationships between $\SP{\kappa,\tau}{\Stationary}$, \linebreak $\SP{\kappa,\tau}{\Unbounded}$ and $\SP{\kappa,\tau}{\Stationary,\Unbounded}$ at an uncountable regular cardinal $\kappa$. Clearly, the latter is the weakest of these principles; it is implied by each of the former. In other words:
\begin{equation}
\label{eqn:splittingnumbers stub}
\mathfrak{s}_{\Stationary,\Unbounded}(\kappa)\le \min(\mathfrak{s}_{\Stationary}(\kappa),\mathfrak{s}_{\Unbounded}(\kappa)).
\end{equation}
Since the failure of $\SP{\kappa,\kappa}{\Unbounded}$ characterizes the weak compactness of $\kappa$ while the failure of $\SP{\kappa,\kappa}{\Stationary}$ characterizes its ineffability, we have that $\SP{\kappa,\kappa}{\Unbounded}$ implies $\SP{\kappa,\kappa}{\Stationary}$. This would seem to suggest that in general, $\SP{\kappa,\tau}{\Unbounded}$ implies $\SP{\kappa,\tau}{\Stationary}$, that is, that $\mathfrak{s}_{\Stationary}(\kappa)\le\mathfrak{s}(\kappa)$. The following corollary is going in that direction.

\begin{corollary}
\label{cor:StationarySNbelowUnboundedSN}
Let $\kappa$ be a regular uncountable cardinal. If $\mathfrak{s}(\kappa),\mathfrak{s}_\Stationary(\kappa)\le\kappa^+$, then $\mathfrak{s}_\Stationary(\kappa)\le\mathfrak{s}(\kappa)$.
\end{corollary}

\noindent\emph{Note:} The proof gives more information.

\begin{proof} Let $\tau=\mathfrak{s}(\kappa)$.

\emph{Case 1:} $\tau<\kappa$.
Then, by applying Corollary \ref{cor:BothWaysForRegulars} to $\Unbounded_\kappa$, it follows that $2^\tau\ge\kappa$, and $\tau$ is the least such cardinal. But by the same corollary, this time applied to $\Stationary_\kappa$, $\tau$ must also be $\mathfrak{s}_\Stationary(\kappa)$. So in this case, we have that $\mathfrak{s}(\kappa)=\mathfrak{s}_{\Stationary}(\kappa)$.

\emph{Case 2:} $\tau=\kappa$.
Then, $\kappa$ is inaccessible, but not weakly compact, so the principle $\SP{\kappa}{\Unbounded}$ holds, by Corollary \ref{corollary:SplitKappaIffNotWC}. It follows that $\kappa$ is not ineffable, so that $\SP{\kappa}{\Stationary}$ holds by Theorem \ref{theorem:CharacterizationOfIneffability}. But by Corollary \ref{cor:BothWaysForRegulars}, $\SP{\kappa,\tau'}{\Stationary}$ fails for $\tau'<\kappa$, as $\kappa$ is inaccessible. So $\tau=\mathfrak{s}_\Stationary(\kappa)$. Thus, in this case we also have that $\mathfrak{s}(\kappa)=\mathfrak{s}_\Stationary(\kappa)$.

\emph{Case 3:} $\tau=\kappa^+$.
Then $\kappa$ is weakly compact by Corollary \ref{corollary:SplitKappaIffNotWC}. If $\kappa$ is not ineffable, then as in the previous case, $\mathfrak{s}_\Stationary(\kappa)=\kappa<\mathfrak{s}(\kappa)$, and if $\kappa$ is ineffable then we get $\mathfrak{s}_\Stationary(\kappa)\ge\kappa^+$, so by assumption, $\mathfrak{s}_\Stationary(\kappa)=\kappa^+$. Thus, in both cases, $\mathfrak{s}_\Stationary(\kappa)\le\mathfrak{s}(\kappa)$.
\end{proof}


\begin{lemma}
Let $\kappa$ be an uncountable regular cardinal, and let $\mathcal{F}\sub\power{\kappa}$ be an $(\Unbounded,\Unbounded)$-splitting family of cardinality $\tau\le\kappa$. Then $\mathcal{F}$ is also $(\Stationary,\Stationary)$-splitting.
\end{lemma}

\noindent\emph{Note:} In particular, if $\mathfrak{s}(\kappa)\le\kappa$, then $\mathfrak{s}_{\Stationary}(\kappa)\le\mathfrak{s}(\kappa)$.

\begin{proof}
Assume the contrary, and let $S\sub\kappa$ be a stationary set that's not split into stationary sets by $\mathcal{F}$. Thus, for any $X\in\mathcal{F}$, exactly one of $S\cap X$ and $S\ohne X$ is nonstationary, that is, either there is a club $C$ such that $(S\cap C)\sub\kappa\ohne X$ or $(S\cap C)\sub X$.

Let's say that $\mathcal{F}'$ is a flip of $\mathcal{F}$ if there is a function $f:\mathcal{F}\To 2$ such that $\mathcal{F}'=\{X\in\mathcal{F}\st f(X)=0\}\cup\{\kappa\ohne X\st X\in\mathcal{F}\land f(X)=1\}$. Clearly, any flip of $\mathcal{F}$ is also $(\Unbounded,\Unbounded)$-splitting, and furthermore, there is a flip $\mathcal{F}'$ of $\mathcal{F}$ such that for every $X\in\mathcal{F}'$ there is a club $C\sub\kappa$ such that we're always in the second case, that is, $(S\cap C)\sub X$. So we may assume that already $\mathcal{F}$ is like that.

Let's write $\mathcal{F}=\{X_\alpha\st\alpha<\kappa\}$ (this may be an enumeration with repetitions). For $\alpha<\kappa$, let $C_\alpha\sub\kappa$ be club such that $S\cap C_\alpha\sub X_\alpha$. Let $D=\bigtriangleup_{\alpha<\kappa}C_\alpha$. So $D$ is club in $\kappa$, making $S\cap D$ stationary and hence unbounded. Since $\mathcal{F}$ is $(\Unbounded,\Unbounded)$-splitting, there is an $\alpha^*<\kappa$ such that $X_{\alpha^*}$ splits $S\cap D$ into unbounded sets. But by definition of $D$, $D\ohne(\alpha^*+1)\sub C_{\alpha^*}$, and so, $(S\cap D)\ohne(\alpha^*+1)\sub S\cap C_{\alpha^*}\sub X_{\alpha^*}$. As a consequence, $(S\cap D)\ohne X_{\alpha^*}\sub\alpha^*+1$, so $(S\cap D)\ohne X_{\alpha^*}$ is not unbounded in $\kappa$, a contradiction.
\end{proof}

Thus, it is interesting to investigate universes with high splitting numbers, so we end this section with the following result, which shows that the method of pushing $\mathfrak{s}(\kappa)$ up by forcing, originally devised by Kamo (and independently by Miyamoto, according to \cite{Suzuki:OnSplittingNumbers}), and exposed in Zapletal \cite[Lemma 5]{Zapletal:SplittingNumberAtUncountableCards}, actually achieves slightly more. We learned from the (unfortunately) unpublished preprint \cite{Kamo:SplittingNumbers} by Kamo, to which we gained access after observing it ourselves, that the crucial claim $(**)$ in the proof, as well as its relevance for the stationary splitting number, was known to him.

\begin{theorem}
\label{thm:PushingUpStatUnbddSplittingNumber}
Suppose $\kappa$ is Laver-indestructibly supercompact\footnote{By a Laver-indestructibly supercompact cardinal, we mean a supercompact cardinal $\kappa$ whose supercompactness is indestructible under $\kappa$-directed closed forcing. Any supercompact cardinal can be forced to be Laver-indestructibly supercompact, as shown in Laver \cite{Laver:Indestructibility}.} and $\lambda>\kappa^+$ is regular. Then there is a cardinal- and cofinality-preserving forcing extension in which $\kappa$ remains regular and
\[\mathfrak{s}_{\Stationary,\Unbounded}(\kappa)=\sn_\Stationary(\kappa)=\sn_\Unbounded(\kappa)=\lambda\]
\end{theorem}

\begin{proof}
The forcing will be an iteration of ``Prikry-Mathias forcing'' of the form $\bbM=\bbM(U)$, for a normal ultrafilter $U$ on $\kappa$. The proof mimics the usual way of increasing the splitting number by iterating Mathias forcing, as described in \cite[p.~444]{Blass2009:CardinalCharacteristics}.

Conditions in $\bbM$ are of the form $\kla{s,A}$, where $s\in[\kappa]^{{<}\kappa}$, $A\in U$ and $s\sub\min(A)$. The ordering is $\kla{s,A}\le\kla{t,B}$ iff $t\sub s$, $s\ohne t\sub B$ and $A\sub B$. So it is just like Prikry forcing \cite{Prikry:Dissertation}, except that the first component of a condition does not have to be finite, and it is just like Mathias forcing  \cite{Mathias:HappyFamilies}, except that we use $\kappa$ instead of $\omega$.

It is easy to see that $\bbM$ is $\kappa$-directed closed. It is also well-met, for given compatible conditions $\kla{s_0,A_0}$ and $\kla{s_1,A_1}$, they have the weakest common extension $\kla{s_0\cup s_1,A_0\cap A_1}$. Just like Prikry forcing, it is $\kappa^+$-c.c.; actually, just like Prikry forcing, it is $\kappa$-centered.

An $\bbM$-generic filter $G$ gives rise to the Prikry-Mathias set
\[g=\bigcup\{s\st\exists A\in U\ \kla{s,A}\in G\},\]
which behaves much like a Prikry sequence: given any $A\in U$, $g\sub^*A$, i.e., $g\ohne A$ is bounded in $\kappa$ (in Prikry forcing, it would be finite). It follows that

\begin{itemize}
\item[$(*)$] for any $X\in\power{\kappa}^\V$, $g\sub^*X$ or $g\sub^*(\kappa\ohne X)$.
\end{itemize}%
This is because either $X\in U$ or $\kappa\ohne X\in U$.
For our purposes, we also need to know:

\begin{itemize}
\item[$(**)$] In $\V[G]$, $g$ is a stationary subset of $\kappa$.
\end{itemize}

\begin{proof}[Proof of $(**)$. ]
Assume the contrary. Then there are a condition $p=\kla{s,A}\in G$ and an $\bbM$-name $\dot{C}$ such that $p\forces$ ``$\dot{C}$ is a club subset of $\check{\kappa}$ such that $\dot{C}\cap\dot{g}=\leer$.'' Let $\dot{f}$ be an $\bbM$-name such that $p\forces$ ``$\dot{f}$ is the monotone enumeration of $\dot{C}$.'' Define in $\V$ a sequence $\seq{p_\alpha}{\alpha<\kappa}$ of conditions in $\bbM$ of the form $p_\alpha=\kla{s_\alpha,A_\alpha}$ such that
\begin{enumerate}
  \item $p_0\le p$.
  \item If $\alpha<\beta<\kappa$, then $p_\beta\le p_\alpha$.
  \item $p_\alpha$ decides the value of $\dot{f}(\check{\alpha})$. That is, there is a $\xi=\xi_\alpha$ such that $p_\alpha\forces$``$\dot{f}(\check{\alpha})=\check{\xi}$.''
  \item If $\alpha$ is a limit ordinal, then $s_\alpha=\bigcup_{\beta<\alpha}s_\beta$ and $A_\alpha=\bigcap_{\beta<\alpha}A_\beta$. That is, $p_\alpha$ is the weakest lower bound of $\seq{p_\beta}{\beta<\alpha}$.
\end{enumerate}
In the successor step of the construction, given $p_\alpha$, we can let $p_{\alpha+1}$ be any condition extending $p_\alpha$ and deciding $\dot{f}(\check{\alpha})$. The point in the limit step is that if $\alpha$ is a limit ordinal and $p_\alpha$ is defined as described, then $p_\alpha$ decides $\dot{f}\rest\check{\alpha}$, but it then automatically also decides $\dot{f}(\check{\alpha})$ to be the supremum of the prior values of $\dot{f}$, since it forces that $\dot{f}$ enumerates a club. Set
\[A^*=\dintersection_{\alpha<\kappa}A_\alpha.\]
Since $U$ is normal, $A^*\in U$.
Let $\bar{f}:\kappa\To\kappa$ be defined by $\bar{f}(\alpha)=\xi_\alpha$, i.e., $p_\alpha\forces\dot{f}(\check{\alpha})=(\bar{f}(\alpha))\check{}$. So $\bar{f}$ is a continuous, increasing function, and since $\kappa$ is regular, the set $F=\{\alpha<\kappa\st\bar{f}(\alpha)=\alpha\}$ of fixed points is club, and hence in $U$, again by normality of $U$.

Now pick a limit ordinal $\alpha\in A^*\cap F$. We then have that
\[p_\alpha=\kla{s_\alpha,A_\alpha}\forces\dot{f}(\check{\alpha})=\check{\alpha}.\]
Since $\alpha\in A^*$, we have that $\alpha\in\bigcap_{\beta<\alpha}A_\beta=A_\alpha$. So $s_\alpha\sub\alpha\in A_\alpha$. Let
\[p'_\alpha=\kla{s_\alpha\cup\{\alpha\},A_\alpha\ohne(\alpha+1)}.\]
Then $p'_\alpha\le p_\alpha$, but
\[p'_\alpha\forces\check{\alpha}\in\dot{C}\cap\dot{g}.\]
This is a contradiction, since $p'_\alpha\le p\forces\dot{C}\cap\dot{g}=\leer$.
\qedhere{${}_{(**)}$}
\end{proof}

For an uncountable regular cardinal $\tau$, let $\P=\Add(\tau)$ be the usual forcing to add a Cohen subset to $\tau$. So it consists of functions $f:\alpha\To 2$, for $\alpha<\tau$, ordered by reverse inclusion. We will use the following fact, where $A$ is the $\P$-generic subset of $\tau$ added by $\P$.

\begin{enumerate}
\item[$({*}{*}{*})$] $A$ splits every unbounded set $X\in\power{\tau}\cap\V$ into unbounded sets, and it splits every set $S\in\power{\tau}\cap\V$ such that $(S\ \text{is stationary})^\V$ into stationary sets (in the sense of $\V[A]$).
\end{enumerate}

\begin{proof}[Proof of $({*}{*}{*})$. ]
The first part is very easy to see, so we prove only the second part, which is also quite easy. So let $S\in\V$ be stationary in $\tau$, and let $C\in\V[A]$ be club in $\tau$. We have to show that both $(S\cap A)\cap C$ and $(S\ohne A)\cap C$ are nonempty. Let $f:\tau\To C$ be the monotone enumeration of $C$, and let $C=\dot{C}^G$, $f=\dot{f}^G$, where $\dot{C}$ and $\dot{f}$ are $\P$-names and $G$ is the $\P$-generic filter associated to $A$. Let $\dot{A}$ be the canonical name for $A$. Suppose that at least one of $(S\cap A)\cap C$ and $(S\ohne A)\cap C$ were empty, and let $p\in G$ force this, that is, $p$ forces ``$(\check{S}\cap\dot{A})\cap\dot{C}$ is empty'' or $p$ forces ``$(\check{S}\ohne\dot{A})\cap\dot{C}$ is empty.'' Choose $p$ strong enough that $p$ also forces that $\dot{C}$ is club in $\check{\tau}$ and that $\dot{f}$ is the monotone enumeration of $\dot{C}$. As in the proof of $(**)$, build a weakly decreasing, continuous sequence $\seq{p_\alpha}{\alpha<\tau}$ of conditions in $\V$ such that $p_0\le p$ and $p_\alpha||\dot{f}(\check{\alpha})$. Continuity means that when $\alpha<\tau$ is a limit ordinal, then $p_\alpha=\bigcup_{\beta<\alpha}p_\beta$. Define $\bar{f}:\tau\To\tau$ by letting $\bar{f}(\alpha)$ be the unique $\xi<\tau$ such that $p_\alpha\forces\dot{f}(\check{\alpha})=\check{\xi}$. Then $\bar{f}$ increasing and continuous, so $F=\{\alpha<\tau\st\bar{f}(\alpha)=\alpha\}$ is club in $\tau$. Moreover, the function mapping $\alpha<\tau$ to $\dom(p_\alpha)$ is weakly increasing and continuous, so the set $D$ of its fixed points is also club in $\tau$. Now let $\alpha\in S\cap F\cap D$. Let $p^+_\alpha=p_\alpha\cup\{\kla{\alpha,1}\}$ and $p^-_\alpha=p_\alpha\cup\{\kla{\alpha,0}\}$. Then $p^+_\alpha$ forces that $\check{\alpha}\in(\check{S}\cap\dot{A})\cap\dot{C}$ and $p^-_\alpha$ forces that $\check{\alpha}\in(\check{S}\ohne\dot{A})\cap\dot{C}$, but both extend $p$, so this contradicts what $p$ forced.
\qedhere{${}_{({*}{*}{*})}$}
\end{proof}

We now form an iteration of length $\lambda$, $\seq{\P_\alpha}{\alpha\le\lambda}$, $\seq{\dot{\Q}_\alpha}{\alpha<\lambda}$ with ${<}\kappa$ support in which every iterand $\dot{\Q}_\alpha$ is forced by $\P_\alpha$ to be of the form $\bbM(\dot{U}_\alpha)$. Since we use ${<}\kappa$ support and each iterand is forced to be $\kappa$-directed closed, $\kappa$-centered and well-met, each $\P_\alpha$ is $\kappa$-directed closed and $\kappa^+$-c.c., and in particular, $\kappa$ remains supercompact in the corresponding extension, so that a normal ultrafilter on $\kappa$ can be found, keeping the iteration going.

Let $G$ be $\P_\lambda$-generic over $\V$. Since $\P_\lambda$ is $\kappa$-directed closed and $\kappa^+$-c.c., $\P_\lambda$ preserves cardinals and cofinalities.

We claim that in $\V[G]$, there is no family of cardinality less than $\lambda$ that splits stationary subsets of $\kappa$ into unbounded subsets of $\kappa$, showing that
\[\sn_{\Stationary,\Unbounded}(\kappa)\ge\lambda.\]
Suppose $\mathcal{F}$ were such a family. Since $\lambda>\kappa^+$ is regular in $\V[G]$, $\P_\lambda$ is $\kappa^+$-c.c., and we used ${<}\kappa$ support, we can find an $\alpha<\lambda$ such that $\mathcal{F}\in\V[G\rest\alpha]$. Let $g_\alpha$ be the Prikry-Mathias set corresponding to $G_\alpha$. By $(*)$, no element of $\V[G\rest\alpha]$ splits $g_\alpha$ into unbounded sets. Moreover, by $(**)$, $g_\alpha$ is stationary in $\V[G\rest\alpha][G_\alpha]$, and since the tail of the iteration is $\kappa$-directed closed, $g_\alpha$ remains stationary in $\V[G]$. Thus, in $\V[G]$, $g_\alpha$ is a stationary subset of $\kappa$ that is not split into unbounded sets by $\mathcal{F}$, a contradiction.

By the inequality (\ref{eqn:splittingnumbers stub}), it remains to show that in $\V[G]$,
\[\max(\sn_{\Stationary}(\kappa),\sn_{\Unbounded}(\kappa))\le\lambda.\]
We show that in $\V[G]$ there is a family of stationary subsets of $\kappa$, of size $\lambda$, that splits stationary subsets of $\kappa$ into stationary sets and that also splits unbounded subsets of $\kappa$ into unbounded sets. Namely, since the iteration uses ${<}\kappa$ support, for every $\alpha<\lambda$, there is in $\V[G\rest(\alpha+\kappa)]$ a $c_\alpha\sub\kappa$ which is $\Add(\kappa)$-generic over $\V[G\rest\alpha]$. By $({*}{*}{*})$, $c_\alpha$ splits every unbounded subset of $\kappa$ which is in $\V[G\rest\alpha]$ into unbounded sets, and it splits every subset of $\kappa$ in $\V[G\rest\alpha]$ which is stationary in $\V[G\rest\alpha]$ into sets which are stationary in $\V[G\rest(\alpha+\kappa)]$. Since the tail of the iteration is $\kappa$-closed, again, these sets remain stationary in $\V[G]$. And since every subset of $\kappa$ in $\V[G]$ occurs in some $\V[G\rest\alpha]$, this shows that the family $\{c_\alpha\st\alpha<\lambda\}$ is both $(\Unbounded_\kappa,\Unbounded_\kappa)$-splitting and  $(\Stationary_\kappa,\Stationary_\kappa)$-splitting in $\V[G]$.
\end{proof}

\noindent\emph{Note:} it seems likely that the methods of Ben-Neria \& Gitik developed in \cite{BenNeria-Gitik:SplittingNumberAtRegularCardinals} can be used to achieve the situation of the previous theorem using weaker large cardinal assumptions. 
\subsection{Subtle cardinals and Souslin trees}
\label{subsec:Subtles}

Our ability to express coherence as in Lemma \ref{lemma:Coherence=NotSplittingIntoNonempty} immediately puts us in a position to characterize subtle cardinals, originally introduced in \cite{Jensen:CombinatorialProperties}. We emphasize that just in the case of ineffability, the relevant split principles involve splitting into nonempty sets, a class that's not a tail set, and so, these split principles don't translate to statements about splitting numbers.

Recall that by definition, a regular cardinal $\kappa$ is subtle iff for every $\kappa$-list $\vec{d}$ and every club $C\sub\kappa$, there are $\alpha<\beta$, both in $C$, such that $d_\alpha=d_\beta\cap\alpha$.

\begin{lemma}
\label{lem:SubtleCharacterizationI}
An uncountable, regular cardinal $\kappa$ is subtle iff for every club $C\sub\kappa$, $\SP{\kappa}{[C]^2,\Nonempty}$ fails.
\end{lemma}

\begin{proof}
The subtlety of $\kappa$ can be equivalently expressed by saying that for every $\kappa$-list $\vec{d}$ and every club $C\sub\kappa$, there is a two-element subset $\{\alpha,\beta\}$ of $C$ on which $\vec{d}$ coheres. But by Lemma \ref{lemma:Coherence=NotSplittingIntoNonempty}, $\vec{d}$ coheres on $\{\alpha,\beta\}$ iff $\vec{d}$ does not split $\{\alpha,\beta\}$ into nonempty sets. Thus, an uncountable, regular cardinal $\kappa$ is subtle iff for every club $C\sub\kappa$, there is no $\kappa$-list that splits $[C]^2$ into nonempty sets, as wished.
\end{proof}

Note that the spit principles used in the characterization of subtlety involve splitting collections of sets which are not closed under supersets.

It turns out that applying a similar idea to thin lists can be used to characterize when a regular uncountable cardinal $\kappa$ carries a Souslin tree.

\begin{theorem}
The following are equivalent for a regular uncountable cardinal $\kappa$:
\begin{enumerate}[label=\textnormal{(\arabic*)}]
  \item
  \label{item:ThereIsASouslinTree}
  There is a $\kappa$-Souslin tree.
  \item
  \label{item:WeirdSplitCondition}
  There is a thin $\kappa$-list $\vec{d}$ such that
  \begin{enumerate}[label=\textnormal{(\alph*)}]
    \item
    \label{item:SplitKappaSequence}
    $\vec{d}$ witnesses that $\Split{\kappa}$ holds.
    \item
    \label{item:Not2EmptySequence}
    For any unbounded $A\sub\kappa$, $\vec{d}$ is \emph{not} a $\SP{\kappa}{[A]^2,\Nonempty}$-sequence.
  \end{enumerate}
\end{enumerate}
\end{theorem}

\begin{proof}
\ref{item:ThereIsASouslinTree}$\implies$\ref{item:WeirdSplitCondition}: By standard arguments, if there is a $\kappa$-Souslin tree, then there is a normal one that is $2$-splitting, and such a tree is isomorphic to a sequential tree which is a subset of ${}^{{<}\kappa}2$, so assume our tree is like that, and call it $T$. For $\alpha<\kappa$, pick some $s_\alpha:\alpha\To 2$ be in $T$, and let $d_\alpha=\{\xi<\alpha\st s_\alpha(\xi)=1\}$. We will show that $\vec{d}$ satisfies the conditions in \ref{item:WeirdSplitCondition}. To see this, note that $T_{\vec{d}}$ is a downward closed subtree of $T$, of height $\kappa$, and hence it is a $\kappa$-Souslin tree. In particular, it is a $\kappa$-Aronszajn tree, so by Theorem \ref{theorem:ThinListSplitsIffItsTreeIsAronszajn}, $\vec{d}$ is a thin $\kappa$-list witnessing that $\Split{\kappa}$ holds, so that condition \ref{item:SplitKappaSequence} is satisfied. To see that condition \ref{item:Not2EmptySequence} is satisfied, let $A\sub\kappa$ be unbounded. Since $T_{\vec{d}}$ is a $\kappa$-Souslin tree, the set $\{d_\alpha\st\alpha\in A\}\sub T_{\vec{d}}$ is not an antichain. This means that there are $\alpha,\beta\in A$ with $\alpha<\beta$ such that $d_\alpha=d_\beta\cap\alpha$. But then, $\{\alpha,\beta\}\in[A]^2$ cannot be split into two nonempty sets by $\vec{d}$, as in Lemma \ref{lemma:Coherence=NotSplittingIntoNonempty}. Thus, $\vec{d}$ is not a $\SP{\kappa}{[A]^2,\Nonempty}$-sequence, as wished.

\ref{item:WeirdSplitCondition}$\implies$\ref{item:ThereIsASouslinTree}: Let $\vec{d}$ be a $\kappa$-list as in \ref{item:WeirdSplitCondition}. By Theorem \ref{theorem:ThinListSplitsIffItsTreeIsAronszajn}, $T_{\vec{d}}$ is a $\kappa$-Aronszajn tree. We claim that it is a $\kappa$-Souslin tree. To see this, suppose it had an antichain $I$ of size $\kappa$. We define an antichain $I_0$, also of cardinality $\kappa$, based on $I$ but translated to the list $\vec d$, as follows. For $\alpha<\kappa$, let $d^c_\alpha:\alpha\To 2$ be the characteristic function of $d_\alpha$. Then, for every $x\in I$, let $\alpha_x<\kappa$ be least such that $x=d^c_{\alpha_x}\rest\dom(x)$, and set $I_0=\{d^c_{\alpha_x}\st x\in I\}$. Note that if $x\neq y$, $x,y\in I$, then $d^c_{\alpha_x}$ and $d^c_{\alpha_y}$ are incomparable in $T_{\vec{d}}$ because if $d^c_{\alpha_x}\sub d^c_{\alpha_y}$, say, then $x,y\sub d^c_{\alpha_y}$, so that already $x$ and $y$ would be comparable, contradicting that they both belong to $I$. Hence, $I_0$ is an antichain in $T_{\vec{d}}$ of size $\kappa$. Now let $A=\{\alpha_x\st x\in I\}$. Then $A\sub\kappa$ is an unbounded set such that for all $\alpha<\beta$, both in $A$, $d_\alpha\neq d_\beta\cap\alpha$. But since $\vec{d}$ is not a $\SP{\kappa}{[A]^2,\Nonempty}$-sequence, there is $\{\alpha,\beta\}\in[A]^2$ which is not split by $\vec{d}$ into nonempty sets. By Lemma \ref{lemma:Coherence=NotSplittingIntoNonempty}, this means that $\vec{d}$ coheres on $\{\alpha,\beta\}$, that is, that $d_\alpha=d_\beta\cap\alpha$, a contradiction.
\end{proof} 
\section{Split principles at singular cardinals}
\label{sec:singular}

Previously, we have mostly assumed $\kappa$ to be regular when considering split principles at $\kappa$. We will drop this assumption here and see what can be said about the general case.

Recall that when $\kappa$ is regular, Theorem \ref{theorem:ListSplitsIffItHasNoBranch} characterized exactly what it means for a $\kappa$-list $\vec{d}$ to be a $\SP{\kappa}{\Unbounded}$-sequence: this is the case iff $\vec{d}$ has no cofinal branch. When $\kappa$ is singular, one direction of this implication holds, while the other direction fails completely.

\begin{observation}
\label{observation:OGSplitAtSingular}
Let $\kappa$ be a singular cardinal.
\begin{enumerate}[label=\textnormal{(\arabic*)}]
  \item
  \label{item:(1)=>(2)}
  If $\vec{d}$ is a $\SP{\kappa}{\Unbounded}$-sequence, then $\vec{d}$ has no cofinal branch.
  \item
  \label{item:(2)not=>(1)}
  There is a thin $\kappa$-list $\vec{d}$ that does not have a cofinal branch and isn't a $\SP{\kappa}{\Unbounded}$-sequence.
\end{enumerate}
\end{observation}

\begin{proof}
For \ref{item:(1)=>(2)}, the proof of this implication in Theorem \ref{theorem:ListSplitsIffItHasNoBranch} did not use the regularity assumption on $\kappa$ and goes through.

For \ref{item:(2)not=>(1)}, let $f:\cf(\kappa)\To\kappa$ be monotone and cofinal, with $f(0)\ge\cf(\kappa)$. Define $g:\kappa\To\cf(\kappa)$ by setting $g(\alpha)=\min\{\bar{\alpha}<\cf(\kappa)\st\alpha<f(\bar{\alpha})\}$. Define a $\kappa$-list $\vec{d}$ by setting: $d_\alpha=\{g(\alpha)\}\cap\alpha$. It is easy to see that $T_{\vec{d}}$ has no cofinal branch, so $\vec{d}$ has no cofinal branch. And $\vec{d}$ is not a $\SP{\kappa}{\Unbounded}$ sequence, because for every $\beta<\kappa$, for all sufficiently large $\alpha<\kappa$, $\beta\notin d_\alpha$. It is also clear that $\vec{d}$ is thin.
\end{proof}

In contrast, the characterization given in Theorem \ref{theorem:GeneralABranchesForKappaListsAndGeneralIneffability} of what it means for a $\kappa$-list $\vec{d}$ to be a $\SP{\kappa}{\calA,\Nonempty}$-sequence is very robust and works for singular $\kappa$ as well: this is the case iff $\vec{d}$ does not cohere on any $A\in\calA$, i.e., iff $\vec{d}$ has no $\calA$ branch.

Approaching the connections between split principles and splitting numbers, it is worth pointing out that the original concept of the splitting number at an uncountable cardinal $\kappa$ was generalized from the case $\kappa=\omega$, where a splitting family is a subfamily of $[\omega]^\omega$ which splits each member of $[\omega]^\omega$ into two members of $[\omega]^\omega$. When considering the version at a regular $\kappa$, the set of $\kappa$-sized subsets of $\kappa$ is the same as the collection of all unbounded subsets of $\kappa$. In the singular case, however, these collections are not the same, and both lead to viable versions of the splitting number at a singular cardinal. Traditionally, the $[\kappa]^\kappa$ way was chosen, as in Suzuki \cite{Suzuki:OnSplittingNumbers} and Zapletal \cite{Zapletal:SplittingNumberAtUncountableCards}, but let us look at all the split principles at a singular $\kappa$ that naturally arise in this way:
$\SP{\kappa,\tau}{[\kappa]^\kappa}$, $\SP{\kappa,\tau}{\Unbounded}$ and $\SP{\kappa,\tau}{[\kappa]^\kappa,\Unbounded}$. Note that $\SP{\kappa,\tau}{\Unbounded,[\kappa]^\kappa}$ is obviously inconsistent, and the following implications are evident (see the comments after Definition \ref{definition:GeneralNonsenseOnKappaSplitting}):
\begin{align}
\label{eq:EasyImplications}
\SP{\kappa,\tau}{[\kappa]^\kappa}&\implies\SP{\kappa,\tau}{[\kappa]^\kappa,\Unbounded}\quad \text{and}\\ \notag\SP{\kappa,\tau}{\Unbounded}&\implies\SP{\kappa,\tau}{[\kappa]^\kappa,\Unbounded}.
\end{align}
We will complete the picture later. First, let us observe a corollary of Lemma \ref{lem:AbstractPowerLifting} (which did not assume regularity).

\begin{corollary}
\label{cor:KappaVeryAccessibleImpliesStrongSplit}
If $\kappa$ is a cardinal and there is a $\tau<\cf(\kappa)$ with $2^\tau\geq\kappa$, then
\begin{enumerate}[label=\textnormal{(\arabic*)}]
\item
\label{item:-unbounded-}
$\SP{\kappa,\tau}{\Unbounded}$ holds,
\item
\label{item:-maxcard-}
$\SP{\kappa,\tau}{[\kappa]^\kappa}$ holds, and
\item
\label{item:-stationary-}
$\SP{\kappa,\tau}{\Stationary}$ holds, assuming that $\cf(\kappa)>\omega$.
\end{enumerate}
\end{corollary}

\begin{proof}
With $\calA=\Unbounded_\kappa$, or $\calA=[\kappa]^\kappa$, or $\calA=\Stationary_\kappa$, and $\mathcal{I}=\power{\kappa}\ohne\calA$ we have that $\mathcal{I}$ is a $\cf(\kappa)$-complete (and hence $\tau^+$-complete) ideal that contains all the singletons, so that Lemma \ref{lem:AbstractPowerLifting} applies and immediately gives the result.
\end{proof}

Next, we aim to investigate situations in which split principles transfer from the cofinality of a cardinal to the cardinal itself. With this in mind, we prove a simple, abstract lemma.

\begin{lemma}
\label{lem:AbstractLifting}
Let $\kappa$ be a singular cardinal, $\bkappa=\cf(\kappa)$, and $\tau$ a cardinal.
Let $\calA,\calB\sub\power{\kappa}$ and $\bar{\calA},\bar{\calB}\sub\power{\bkappa}$.
Let $g:\kappa\To\bkappa$ be a partial function, and suppose that the following assumptions are satisfied:
\begin{enumerate}[label=\textnormal{(A\arabic*)}]
	\item
	\label{item:SplitHoldsDownstairs}
	$\SP{\bkappa,\tau}{\bar{\calA},\bar{\calB}}$ holds.
	\item
	\label{item:Deflation}
	If $A\in\calA$, then $g``A\in\bar{\calA}$.
	\item
	\label{item:Closure under supersets}
	$\calB$ is closed under supersets.
	\item
	\label{item:Specializing}
	Let $A\in\calA$, and let $\bA=g``A$. Then there is an $X\sub\bA$ with $X\in\bar{\calA}$ such that whenever $\bar{B}\sub X$ and $\bar{B}\in\bar{\calB}$, then $A\cap(g^{-1}``\bar{B})\in\calB$.
\end{enumerate}
Then $\SP{\kappa,\tau}{\calA,\calB}$ holds.
\end{lemma}

\begin{proof}
By \ref{item:SplitHoldsDownstairs}, let $\vec{\bar{d}}=\seq{\bar{d}_\alpha}{\alpha<\bkappa}$ witness that $\SP{\bkappa,\tau}{\bar{\calA},\bar{\calB}}$ holds. Define the sequence $\vec{d}=\seq{d_\alpha}{\alpha<\kappa}$ by setting $d_\alpha=\bar{d}_{g(\alpha)}$ (let $d_\alpha$ be the empty set if $g(\alpha)$ is undefined). We claim that $\vec{d}$ witnesses that $\SP{\kappa,\tau}{\calA,\calB}$ holds.

To this end, let $A\in\calA$ be given, and set $\bA=g``A$. By \ref{item:Deflation}, $\bA\in\bar{\calA}$. Let $X\sub\bar{A}$, $X\in\bar{\calA}$ be as in \ref{item:Specializing}. Now, let $\beta<\tau$ be such that both $\tilde{X}^+_{\beta,\vec{\bar{d}}}$ and $\tilde{X}^-_{\beta,\vec{\bar{d}}}$ are in $\bar{\calB}$. We want to show that $\beta$ splits $A$ into elements of $\calB$ via the sequence $\vec{d}$.

Since $X$ satisfies the condition in \ref{item:Specializing}, we know that both $A\cap(g^{-1}``\tilde{X}^+_{\beta,\vec{\bar{d}}})$ and $A\cap(g^{-1}``\tilde{X}^-_{\beta,\vec{\bar{d}}})$ are in $\calB$.

But we have that $A\cap(g^{-1}``\tilde{X}^+_{\beta,\vec{\bar{d}}})\sub \tilde{A}^+_{\beta,\vec{d}}$, because if $\alpha\in A\cap(g^{-1}``\tilde{X}^+_{\beta,\vec{\bar{d}}})$, then $\alpha\in A$ and $g(\alpha)\in \tilde{X}^+_{\beta,\vec{\bar{d}}}$, which implies that $\beta\in\bar{d}_{g(\alpha)}=d_\alpha$. Thus, $\alpha\in A$ and $\beta\in d_\alpha$, which means that $\alpha\in\tilde{A}^+_\beta$. So, since by assumption \ref{item:Closure under supersets}, $\calB$ is closed under supersets, it follows that $\tilde{A}^+_{\beta,\vec{d}}\in\calB$.

Similarly, $A\cap(g^{-1}``\tilde{X}^-_{\beta,\vec{\bar{d}}}) \sub\tilde{A}^-_{\beta,\vec{d}}\in\calB$.

This shows that $\vec{d}$ witnesses that $\SP{\kappa,\tau}{\calA,\calB}$ holds, as claimed.
\end{proof}


\begin{lemma}
\label{lem:ConcreteLifting}
Let $\kappa$ be a singular cardinal, $\bkappa=\cf(\kappa)$, and let $\tau$ be a cardinal.
\begin{enumerate}[label=\textnormal{(\arabic*)}]
    \item
    \label{item:UnboundedLifting}
    If $\SP{\bkappa,\tau}{\Unbounded}$ holds, then so does $\SP{\kappa,\tau}{\Unbounded}$.
    \item
    \label{item:MaxCardLifting}
    If $\SP{\bkappa,\tau}{[\bkappa]^{\bkappa}}$ holds, then so does $\SP{\kappa,\tau}{[\kappa]^\kappa}$.
    \item
    \label{item:StationaryLifting}
    If $\SP{\bkappa,\tau}{\Stationary}$ holds, then so does $\SP{\kappa,\tau}{\Stationary}$.
\end{enumerate}
\end{lemma}

\noindent\emph{Note:} In other words, $\mathfrak{s}_{\Unbounded}(\kappa)\le\mathfrak{s}_{\Unbounded}(\cf(\kappa))$,
$\mathfrak{s}_{[\kappa]^\kappa}(\kappa)\le\mathfrak{s}_{[\cf(\kappa)]^{\cf(\kappa)}}(\cf(\kappa))$ and for $\cf(\kappa)>\omega$, $\mathfrak{s}_{\Stationary}(\kappa)\le\mathfrak{s}_{\Stationary}(\cf(\kappa))$.

\begin{proof}
We will use Lemma \ref{lem:AbstractLifting} in each case.

For \ref{item:UnboundedLifting}: let $f:\bkappa\To\kappa$ be strictly increasing, cofinal and continuous, with $f(0)=0$, say. Define $g:\kappa\To\bkappa$ by setting $g(\alpha)=\min\{\balpha<\bkappa\st\alpha<f(\balpha+1)\}$. That is, $g(\alpha)$ is the unique $\balpha$ such that $f(\balpha)\le\alpha<f(\balpha+1)$. We use Lemma \ref{lem:AbstractLifting} with this function $g$, $\calA=\calB=\Unbounded_\kappa$ and $\bar{\calA}=\bar{\calB}=\Unbounded_{\bkappa}$. Condition \ref{item:SplitHoldsDownstairs} holds by assumption. It is obvious that if $A\sub\kappa$ is unbounded, then $g``A\sub\bkappa$ is unbounded in $\bkappa$, that is, condition \ref{item:Deflation} holds. Condition \ref{item:Closure under supersets} is trivially satisfied, and condition \ref{item:Specializing} holds with $X=\bar{A}=g``A$.

For \ref{item:MaxCardLifting}: let $f$ and $g$ be as above. We let $\calA=\calB=[\kappa]^\kappa$ and $\bar{\calA}=\bar{\calB}=[\bkappa]^{\bkappa}$. Let us check the conditions of Lemma \ref{lem:AbstractLifting}. Condition \ref{item:SplitHoldsDownstairs} again holds by assumption. If $A\sub\kappa$ has cardinality $\kappa$, then it is unbounded, and so, $g``A\sub\bkappa$ is unbounded in $\bkappa$, and hence is in $[\bkappa]^{\bkappa}=\Unbounded_\bkappa$, as $\bkappa$ is regular. So condition \ref{item:Deflation} is satisfied. Condition \ref{item:Closure under supersets} is obvious. It remains to check condition \ref{item:Specializing}. So let $A\in\calA$, and set $\bar{A}=g``A$. So $\bar{A}$ is an unbounded subset of $\bkappa$. Since $A$ has size $\kappa$, it is straightforward to find a set $X\sub\bar{A}$, unbounded in $\bkappa$, such that the restriction of the function $\balpha\mapsto\card{A\cap[f(\balpha),f(\balpha+1))}$ to $X$ is strictly increasing and cofinal in $\kappa$. Then, if $\bar{B}\sub X$ is unbounded, clearly, $A\cap(g^{-1}``\bar{B})$ has cardinality $\kappa$.

For \ref{item:Stationary}:
By our assumption, $\SP{\bkappa,\tau}{\Stationary}$ holds. Note that this implies that $\bkappa$ is uncountable, because if $\bkappa=\omega$, then the stationary subsets of $\bkappa$ are those that contain a tail, but no set that contains a tail can be split into two such sets. So $\kappa$ has uncountable cofinality. Let $f:\bkappa\To\kappa$ be as before, let $g=f^{-1}$, and set $\calA=\calB=\Stationary_\kappa$ and $\bar{\calA}=\bar{\calB}=\Stationary_\bkappa$. Of the conditions of Lemma \ref{lem:AbstractLifting}, \ref{item:SplitHoldsDownstairs} and \ref{item:Closure under supersets} are again clear. For the remaining conditions, note that $\ran(f)$ is club in $\kappa$. So if $A\sub\kappa$ is stationary, then so is $A\cap\ran(f)$, and it is easy to see that $f^{-1}``A=g``A$ is stationary in $\bkappa$, showing \ref{item:Deflation}.
To verify condition \ref{item:Specializing}, let $A\in\calA$, and let $\bar{A}=g``A=f^{-1}``A$. Set $X=\bar{A}$. Now if $\bar{B}\sub X$ is stationary, then $f``\bar{B}=g^{-1}``\bar{B}$ is stationary in $\kappa$. But of course, since $\bar{B}\sub X=\bar{A}=f^{-1}``A$, $f``\bar{B}=A\cap f``\bar{B}=A\cap g^{-1}``\bar{B}$. So this latter set is stationary in $\kappa$, as wished.
\end{proof}

\begin{corollary}
\label{cor:WeaklyCompactCofinality}
Let $\kappa$ be a singular cardinal.
\begin{enumerate}[label=\textnormal{(\arabic*)}]
\item
\label{item:CFWC}
If $\cf(\kappa)$ is uncountable but not weakly compact, then $\Split{\kappa}$ and $\SP{\kappa}{[\kappa]^\kappa}$ hold.
\item
\label{item:CFIE}
If $\cf(\kappa)=\omega$, then $\SP{\kappa}{\Stationary}$ fails. If $\cf(\kappa)$ is uncountable but not ineffable, then $\SP{\kappa}{\Stationary}$ holds.
\end{enumerate}
\end{corollary}

\begin{proof}
For \ref{item:CFWC}, under the stated assumptions, $\SP{\cf(\kappa)}{\Unbounded}$ (which is the same as $\SP{\cf(\kappa)}{[\cf(\kappa)]^{\cf(\kappa)}}$) holds by Corollary \ref{corollary:SplitKappaIffNotWC}, so the claim follows from Lemma \ref{lem:ConcreteLifting}.\ref{item:UnboundedLifting} and \ref{item:MaxCardLifting}.

For \ref{item:CFIE}, if $\cf(\kappa)=\omega$, then a set $A\sub\kappa$ is stationary iff it contains some tail $[\alpha,\kappa)$. Clearly, $A$ cannot be split into two sets of this form, so $\SP{\kappa}{\Stationary}$ fails. If $\cf(\kappa)$ is uncountable but not ineffable, then $\SP{\cf(\kappa)}{\Stationary}$ holds by \ref{theorem:CharacterizationOfIneffability}, and this implies the claim by Lemma \ref{lem:ConcreteLifting}.\ref{item:StationaryLifting}.
\end{proof}

\begin{corollary}
\label{cor:TransferDependingOnPower}
Let $\kappa$ be a singular cardinal. Then:
\begin{enumerate}[label=\textnormal{(\arabic*)}]
    \item
    \label{item:UnboundedLiftingToPower}
    $\SP{\kappa,2^{\cf(\kappa)}}{\Unbounded}$ holds. \\
    In particular, if $2^{\cf(\kappa)}\le\kappa$, then $\SP{\kappa}{\Unbounded}$ holds.
    \item
    \label{item:MaxCardLiftingToPower}
    $\SP{\kappa,2^{\cf(\kappa)}}{[\kappa]^\kappa}$ holds. \\
    In particular, if $2^{\cf(\kappa)}\le\kappa$, then $\SP{\kappa}{[\kappa]^\kappa}$ holds.
    \item
    \label{item:StationaryLiftingToPower}
    $\SP{\kappa,2^{\cf(\kappa)}}{\Stationary}$ holds, assuming that $\cf(\kappa)>\omega$. \\
    In particular, if in addition $2^{\cf(\kappa)}\le\kappa$, then $\SP{\kappa}{\Stationary}$ holds.
\end{enumerate}
\end{corollary}

\begin{proof}
This is because, letting $\bkappa=\cf(\kappa)$, $\SP{\bkappa,2^\bkappa}{\Unbounded}$, $\SP{\bkappa,2^\bkappa}{[\bkappa]^\bkappa]}$ and $\SP{\bkappa,2^\bkappa}{\Stationary}$ hold, as in each case, there is a splitting family of size $2^\bkappa$. The claims now follow from Lemma \ref{lem:ConcreteLifting} (using the fact that the families under consideration are tail sets).
\end{proof}

We now turn towards transferring the split principle downward from a cardinal to its cofinality. It turns out that \ref{item:UnboundedLifting} and \ref{item:StationaryLifting} of Lemma \ref{lem:ConcreteLifting} can be reversed. We will point out later that this is not always so for \ref{item:MaxCardLifting}.

\begin{lemma} \label{lem:Lowering}
Let $\bkappa,\kappa$ and $\tau$ be cardinals, $\bar{\calA},\bar{\calB}\sub\power{\bkappa}$, $\calA,\calB\sub\power{\kappa}$ and $f:\bkappa\To\kappa$ an injective function such that for all $\bar{A}\in\bar{\calA}$, $f``\bar{A}\in\calA$, and further, if $B\in\calB$ and $B\sub f``\bar{A}$, then $f^{-1}``B\in\bar{\calB}$. Then $\SP{\kappa,\tau}{\calA,\calB}\implies\SP{\bkappa,\tau}{\bar{\calA},\bar{\calB}}$.

In particular, if $\bkappa=\cf(\kappa)$, then
\begin{enumerate}[label=\textnormal{(\arabic*)}]
\item
\label{item:UnboundedLowering}
If $\SP{\kappa,\tau}{\Unbounded}$ holds, then so does $\SP{\bkappa,\tau}{\Unbounded}$.
\item
\label{item:StationaryLowering}
If $\bkappa\ge\omega_1$ and $\SP{\kappa,\tau}{\Stationary}$ holds, then so does $\SP{\bkappa,\tau}{\Stationary}$.
\end{enumerate}
\end{lemma}

%


\begin{proof}
To prove the general claim, let $\vec{d}=\seq{d_\alpha}{\alpha<\kappa}$ be a $\SP{\kappa,\tau}{\calA,\calB}$-sequence. Define $\vec{\bar{d}}=\seq{\bar{d}_\balpha}{\balpha<\bkappa}$ by $\bar{d}_\balpha=d_{f(\balpha)}$. Then $\vec{\bar{d}}$ is a $\SP{\bkappa,\tau}{\bar{\calA},\bar{\calB}}$-sequence, because given $\bar{A}\in\bar{\calA}$, $A=f``\bar{A}\in\calA$, so there is a $\beta<\tau$ such that both $\tilde{A}^+_{\beta,\vec{d}}$ and $\tilde{A}^-_{\beta,\vec{d}}$ are in $\calB$. Both of these sets are subsets of $f``\bA$, so by assumption, their pointwise preimages are in $\bar{\calA}$. But $f^{-1}``\tilde{A}^+_{\beta,\vec{d}}=\tilde{\bar{A}}^+_{\beta,\vec{\bar{d}}}$ and $f^{-1}``\tilde{A}^-_{\beta,\vec{d}}=\tilde{\bar{A}}^-_{\beta,\vec{\bar{d}}}$, so $\beta$ splits $\bar{A}$ with respect to $\vec{\bar{d}}$.

Claim \ref{item:UnboundedLowering} now follows by letting $f:\bkappa\To\kappa$ be an increasing, cofinal function, and claim \ref{item:StationaryLowering} follows by choosing $f$ so that it is in addition continuous.
\end{proof}

Thus, the computation of $\sn_{\Unbounded}(\kappa)$ and $\sn_{\Stationary}(\kappa)$ for singular $\kappa$ completely reduces to the regular case.

\begin{corollary}
Let $\kappa$ be a singular cardinal.
\begin{enumerate}[label=\textnormal{(\arabic*)}]
\item $\mathfrak{s}_{\Unbounded}(\kappa)=\mathfrak{s}_{\Unbounded}(\cf(\kappa))=\mathfrak{s}(\cf(\kappa))$.
\item If $\cf(\kappa)\ge\omega_1$, then
$\mathfrak{s}_{\Stationary}(\kappa)=\mathfrak{s}_{\Stationary}(\cf(\kappa))$.
\end{enumerate}
\end{corollary}

\begin{corollary}
For cardinals $\kappa$ and $\tau$, $\SP{\kappa,\tau}{\Unbounded}$ implies $\SP{\kappa,\tau}{[\kappa]^\kappa}$. In other words, $\mathfrak{s}(\kappa)\le\mathfrak{s}_\Unbounded(\kappa)$.
\end{corollary}

\begin{proof}
Let $\bkappa=\cf(\kappa)$. Then $\SP{\kappa,\tau}{\Unbounded}$ implies $\SP{\bkappa,\tau}{\Unbounded}$, which is equivalent to $\SP{\bkappa,\tau}{[\bkappa]^\bkappa}$, and that implies $\SP{\kappa,\tau}{[\kappa]^\kappa}$.
\end{proof}

Thus, we have completed the picture begun in (\ref{eq:EasyImplications}), for a singular cardinal $\kappa$:
\begin{equation}\label{eq:LessEasyImplications}
\SP{\kappa,\tau}{\Unbounded}\implies\SP{\kappa,\tau}{[\kappa]^\kappa}\implies\SP{\kappa,\tau}{[\kappa]^\kappa,\Unbounded}
\end{equation}
(and $\SP{\kappa,\tau}{\Unbounded,[\kappa]^\kappa}$ fails always).

Recall that the width of a sequence, as defined in Definition \ref{def:RangeWidthSpread}, is the size of the range of the sequence. The lifted split sequences in Lemma \ref{lem:ConcreteLifting} have width $\bkappa=\cf(\kappa)$ (as can be seen from their definition in the proof of Lemma \ref{lem:AbstractLifting}, the width of the lifted sequence is going to be at most $\bkappa$, but by Observation \ref{obs:Width>=Completeness} the width must also be at least $\bkappa$). The same is true of the pushed down split sequences of Lemma \ref{lem:Lowering} (in this case, since the pushed down sequence is a $\bkappa$-sequence, it is trivial that its width is at most $\bkappa$, but again by Observation \ref{obs:Width>=Completeness}, it must also be at least $\bkappa$.) As mentioned before, the downward transfer Lemma \ref{lem:Lowering} does not work for $\SP{\kappa,\tau}{[\kappa]^\kappa}$, that is, this principle does not imply $\SP{\bkappa,\tau}{[\bkappa]^\bkappa}$. This is the obstacle to answering the main question of \cite{Zapletal:SplittingNumberAtUncountableCards}, whether it is consistent that $\sn_{[\aleph_\omega]^{\aleph_\omega}}(\aleph_\omega)>\aleph_\omega$, or more generally, that $\sn_{[\kappa]^\kappa}(\kappa)>\kappa$, for a singular cardinal $\kappa$. Thus, it is worth looking more closely at the conditions under which a downward transfer of $\SP{\kappa,\tau}{[\kappa]^\kappa}$-sequences is possible. The concept of thinness, as defined for Split sequences in Definition \ref{def:TwoCardinalSplit}, will yet again prove its usefulness.

%
%

\begin{lemma}
\label{lem:DownwardTransferForSvelteSequences}
Let $\kappa$ be a singular cardinal, and let $\bkappa=\cf(\kappa)$. Then thin $\SP{\kappa,\tau}{[\kappa]^\kappa}$ implies $\SP{\bkappa,\tau}{[\bkappa]^\bkappa}$.
\end{lemma}

\noindent\emph{Note:} A slightly weaker condition than thinness suffices.

\begin{proof}
Let $\vec{d}$ be a thin $\SP{\kappa,\tau}{[\kappa]^\kappa}$-sequence. It then satisfies the following condition, alluded to in the note:

\begin{itemize}
\item[$(*)$]
For every $\theta<\tau$ and every $\lambda<\kappa$, there is an $a\sub\theta$ such that the set $\{\alpha<\kappa\st d_\alpha\cap\theta=a\}$ has cardinality at least $\lambda$.
\end{itemize}

To see this, let $\theta<\tau$, and let $\gamma<\kappa$ be the cardinality of $D_\theta=\{d_\alpha\cap\theta\st\alpha<\kappa\}$. For $a\in D_\theta$, let $s(a)=s_\theta(a)=\{\alpha<\kappa\st d_\alpha\cap\theta=a\}$. It cannot be that there is a cardinal $\lambda<\kappa$ such that for all $a\in D_\theta$, $\card{s_\theta(a)}<\lambda$, as in that case, $\kappa=\bigcup_{a\in D_\theta}s_\theta(a)$ would have cardinality at most $\gamma\cdot\lambda$, a contradiction.

Now let $g:\bkappa\To\kappa$ be increasing and cofinal. For $\theta<\tau$ and $a\sub\theta$, let
\[s_\theta(a)=\{\alpha<\kappa\st d_\alpha\cap\theta=a\},\]
and define $f:\bkappa\To\kappa$ so that for all $\balpha<\bkappa$,
\[\card{s_{g(\balpha)}(d_{f(\balpha)}\cap g(\balpha))}\ge g(\balpha).\]
By $(*)$, such an $f(\balpha)$ always exists. Define
\(\bar{d}_\balpha=d_{f(\balpha)}\)
for $\balpha<\bkappa$. We claim that $\vec{\bar{d}}=\seq{\bar{d}_{\balpha}}{\balpha<\bkappa}$ is a $\SP{\bkappa,\tau}{[\bkappa]^\bkappa}$-sequence, which is the same as a $\SP{\bkappa,\tau}{\Unbounded}$-sequence, since $\bkappa$ is regular.

Let $\bar{A}\sub\bkappa$ be unbounded. For $\balpha\in\bA$, let
\[\tilde{s}(\balpha)=\text{the shortest initial segment of $s_{g(\balpha)}(d_{f(\balpha)}\cap g(\balpha))$ that has size $g(\balpha)$}\]
and define
\[A=\bigcup_{\balpha\in\bA}\tilde{s}(\balpha).\]
Since $\card{\tilde{s}(\balpha)}=g(\balpha)$, $A$ has cardinality $\kappa$, and hence, there is a $\beta<\tau$ which splits $A$ into elements of $[\kappa]^\kappa$ with respect to $\vec{d}$. We claim that $\beta$ also splits $\bA$ into unbounded subsets of $\bkappa$ with respect to $\vec{\bar{d}}$.

To this end, we show that $\tilde{\bA}^+_{\beta,\vec{\bar{d}}}$ is unbounded in $\bkappa$. So fix a $\bdelta<\bkappa$. We have to find an $\balpha>\bdelta$ in $\tilde{\bA}^+_{\beta,\vec{\bar{d}}}$. Let $\bbeta<\bkappa$ be large enough that $g(\bbeta)>\beta$, and let
\[X=\bigcup_{\balpha\le\bdelta\cup\bbeta}\tilde{s}(\balpha).\]
Since $\card{X}<\kappa$, we can pick $\alpha\in\tilde{A}^+_{\beta,\vec{d}}\ohne X$.
This means that $\alpha\in A$ and $\beta\in d_\alpha$. As $\alpha\in A$, there is an $\balpha\in\bA$ such that $\alpha\in\tilde{s}(\balpha)$. So $d_{f(\balpha)}\cap g(\balpha)=d_\alpha\cap g(\balpha)$. Since by definition, $\bard_\balpha=d_{f(\balpha)}$, this means that
\(\bard_\balpha\cap g(\balpha)=d_\alpha\cap g(\balpha).\)
We know further that $\alpha\notin X$. This implies that $\balpha>\bdelta$ and $\balpha>\bbeta$. The latter implies that $g(\balpha)>g(\bbeta)>\beta$. Thus, we have that
\(\beta\in d_\alpha\cap g(\balpha)=\bard_\balpha\cap g(\balpha)\)
and in particular, $\balpha\in\tilde{\bA}^+_{\beta,\vec{\bard}}$. Since $\balpha>\bdelta$, $\balpha$ is as wished.

An entirely analogous argument shows that $\tilde{\bar{A}}^-_{\beta,\vec{\bar{d}}}$ is unbounded in $\bkappa$.
\end{proof}


\begin{corollary}
\label{cor:ThinSplitImpliesSmallWidth}
Let $\kappa$ be singular, $\bkappa=\cf(\kappa)$, and let $\tau$ be a cardinal. If thin $\SP{\kappa,\tau}{[\kappa]^\kappa}$ holds, then there is a $\SP{\kappa,\tau}{[\kappa]^\kappa}$-sequence of width $\bkappa$.
\end{corollary}
\begin{proof}
This is because by Lemma \ref{lem:DownwardTransferForSvelteSequences}, $\SP{\bkappa,\tau}{[\bkappa]^\bkappa}$ holds, and transferring such a sequence back via Lemma \ref{lem:ConcreteLifting}, using the method of the proof of Lemma \ref{lem:AbstractLifting}, produces a $\SP{\kappa,\tau}{[\kappa]^\kappa}$-sequence of width at most $\bkappa$, but it then must have width $\bkappa$ by Observation \ref{obs:Width>=Completeness}.
\end{proof}

The following corollary is essentially due to Zapletal \cite{Zapletal:SplittingNumberAtUncountableCards}, where the proof is given for the case $\kappa=\aleph_\omega$. It is a direct consequence of the downward transfer of thin split sequences.

\begin{corollary}
If $\kappa$ is a singular strong limit cardinal, then $\sn(\kappa)=\sn(\cf(\kappa))$.
\end{corollary}

\begin{proof}
Let $\bkappa=\cf(\kappa)$. Then by Corollary \ref{cor:TransferDependingOnPower}.\ref{item:MaxCardLiftingToPower}, $\SP{\kappa,2^{\bkappa}}{[\kappa]^\kappa}$ holds. So $\sn(\kappa)\le 2^\bkappa<\kappa$, and similarly, $\sn(\bkappa)\le 2^\bkappa<\kappa$. But for $\tau<\kappa$, $\SP{\kappa,\tau}{[\kappa]^\kappa}$ holds iff $\SP{\bkappa,\tau}{[\bkappa]^\bkappa}$ holds, because of the upward transfer provided by Lemma \ref{lem:ConcreteLifting} and the downward transfer due to Lemma \ref{lem:DownwardTransferForSvelteSequences} which applies, because $\kappa$ is a strong limit cardinal, so that every $\SP{\kappa,\tau}{[\kappa]^\kappa}$-sequence is thin.
\end{proof}

Using a result attributed to Shizuo Kamo as a black box, we can answer two questions we had: the principle $\SP{\kappa}{\Unbounded}$ can fail at a singular cardinal $\kappa$, and in fact we can separate $\SP{\kappa}{\Unbounded}$ from $\SP{\kappa}{[\kappa]^\kappa}$, showing that the first implication in (\ref{eq:LessEasyImplications}) cannot be reversed. Unfortunately, we were not able to obtain access to a proof of Kamo's result.

\begin{corollary}
It is consistent that $\SP{\aleph_\omega}{\Unbounded}$ fails and $\SP{\aleph_\omega}{[\aleph_\omega]^{\aleph_\omega}}$ holds.
\end{corollary}

\begin{proof}
According to \cite[Lemma 21]{Zapletal:SplittingNumberAtUncountableCards},
Kamo showed that it is consistent that $\mathfrak{s}(\aleph_0)>\aleph_\omega$ and $\mathfrak{s}(\aleph_\omega)=\aleph_1$. No reference was given for this result. The only reference to work of Kamo given in \cite{Zapletal:SplittingNumberAtUncountableCards} is the unpublished preprint \cite{Kamo:SplittingNumbers}, which we were finally able to retrieve, but it does not contain this result. In any case, in a model like this, $\mathfrak{s}_{\Unbounded}(\aleph_\omega)=\mathfrak{s}(\aleph_0)>\aleph_\omega$, so $\SP{\aleph_\omega}{\Unbounded}$ fails. Moreover, since $\mathfrak{s}(\aleph_\omega)=\aleph_1$, it follows that $\SP{\aleph_\omega}{[\aleph_\omega]^{\aleph_\omega}}$ holds.
\end{proof}

Note that the model of the previous proof shows that \ref{item:MaxCardLifting} of Lemma \ref{lem:ConcreteLifting} cannot be reversed, for in that model, $\SP{\aleph_\omega,\aleph_\omega}{[\aleph_\omega]^{\aleph_\omega}}$ holds, yet $\SP{\aleph_0,\aleph_\omega}{[\aleph_0]^{\aleph_0}}$ fails. This is why in the following theorem, our methods don't allow us to conclude the consistency of the failure of $\SP{\kappa}{[\kappa]^\kappa}$ at a singular $\kappa$  (i.e., $\sn(\kappa)>\kappa$), which is the main open problem of \cite{Zapletal:SplittingNumberAtUncountableCards}. But we are able to resolve the corresponding questions about $\SP{\kappa}{\Unbounded}$, $\SP{\kappa}{\Stationary}$ and \emph{thin} $\SP{\kappa}{[\kappa]^\kappa}$.

\begin{theorem}
\label{thm:SP(stat)CanFailAtSingular}
Assuming the consistency of a supercompact cardinal, it is consistent that there is a singular cardinal $\kappa$ (of uncountable cofinality) such that $\SP{\kappa}{\Stationary}$, $\SP{\kappa}{\Unbounded}$ and thin $\SP{\kappa}{[\kappa]^\kappa}$ fail.

In fact, for any $\tau$, we can force that $\SP{\kappa,\tau}{\Stationary}$, $\SP{\kappa,\tau}{\Unbounded}$ and thin $\SP{\kappa,\tau}{[\kappa]^\kappa}$ fail.
\end{theorem}

\begin{proof}
Starting in a model with a Laver-indestructible supercompact cardinal $\bkappa$, and given $\tau$, let us verify the last claim first. Pick any singular cardinal $\kappa>\bkappa$ with $\cf(\kappa)=\bkappa$. Let $\lambda>\kappa,\tau$ be regular. By Theorem \ref{thm:PushingUpStatUnbddSplittingNumber}, there is a cofinality-preserving forcing extension $\V[G]$ in which $\sn_{\Stationary,\Unbounded}(\bkappa)=\lambda>\tau$. This means that $\SP{\bkappa,\tau}{\Stationary,\Unbounded}$ fails, and this in turn implies further that $\SP{\bkappa,\tau}{\Stationary}$ fails. It also implies that $\SP{\bkappa,\tau}{\Unbounded}$ fails, which is the same as $\SP{\bkappa,\tau}{[\bkappa]^\bkappa}$.
In $\V[G]$, it is still the case that $\cf(\kappa)=\bkappa>\omega$, and so, by (the contrapositive of) Lemma \ref{lem:Lowering}, it must be that $\SP{\kappa,\kappa}{\Stationary}$, which is equivalent to $\SP{\kappa}{\Stationary}$, fails. For the same reason, $\SP{\kappa}{\Unbounded}$ fails in $\V[G]$. And by (the contrapositive of) Lemma \ref{lem:DownwardTransferForSvelteSequences}, there can't be a thin $\SP{\kappa,\tau}{[\kappa]^\kappa}$-sequence in $\V[G]$.

The first claim follows by letting $\tau=\kappa$, using Observation \ref{observation:IndependentOfInitialSegmentsGivesSplitEquivalentToGeneralSplit}, which is justified as all the sets we split into here are tail sets.
\end{proof}

Let us add that if $\calF$ is an $(\calA,\calB)$-splitting family, where $\calA,\calB\sub\power{\kappa}$, then we can make sense of $\calF$ being thin. Namely, if $\card{\calF}=\tau$ and we fix an enumeration $\calF=\{A_\alpha\st\alpha<\tau\}$, and define $d_\alpha=\{\xi<\tau\st\alpha\in A_\xi\}$ for $\alpha<\kappa$, then we can define that $\calF$ is thin iff $\vec{d}$ is thin. We can also define the width of $\calF$ to be the width of $\vec{d}$. So in the model of the previous theorem, the smallest size of a thin $[\kappa]^\kappa$-splitting family greater than $\tau$. Lemma \ref{lem:ConcreteLifting} shows that there is always a thin $[\kappa]^\kappa$-splitting family, in fact there is always one of width $\cf(\kappa)$.

There are many open questions about splitting numbers of uncountable cardinals, but one that hasn't been asked before seems to be the following.

\begin{question}
Let $\kappa$ be an uncountable regular cardinal. Must it be that \[\sn_\Stationary(\kappa)\le\sn(\kappa)\text{?}\]%
In other words, is it true that
\[\SP{\kappa,\tau}{\Unbounded}\implies\SP{\kappa,\tau}{\Stationary}\text{?}\]
\end{question}%
So this question asks whether Corollary \ref{cor:StationarySNbelowUnboundedSN} can be extended past $\kappa^+$.

\bibliography{Spliterature}
\bibliographystyle{plain}
\end{document}